\documentclass[11pt]{article}
\usepackage{smile}

\usepackage{fullpage}
\usepackage{url}
\usepackage{lscape}
\usepackage{bigints}
\usepackage{framed}
\usepackage{mdframed}
\usepackage{enumerate}
\usepackage[inline]{enumitem}
\usepackage[T1]{fontenc}
\usepackage{moresize}
\usepackage{bm}
\usepackage{bbm}
\usepackage{dsfont}
\usepackage{amsmath}
\usepackage{amssymb}
\usepackage{amsthm}
\usepackage{amsfonts}
\usepackage{stmaryrd}
\usepackage{array}
\usepackage{mathrsfs}
\usepackage{mathtools} 
\usepackage{extarrows}
\usepackage{stackrel}
\usepackage{relsize,exscale}
\usepackage{scalerel}
\usepackage{colortbl}
\usepackage[nodisplayskipstretch]{setspace}
\usepackage{color}
\usepackage[usenames,dvipsnames]{xcolor}
\usepackage{cancel}
\usepackage{soul}
\usepackage{undertilde}
\usepackage{xfrac}
\usepackage{siunitx}
\usepackage{graphicx}
\usepackage{float}
\usepackage{rotating}
\usepackage{subcaption}
\usepackage{overpic}
\usepackage[all]{xy}
\usepackage{tikz}
\usetikzlibrary{arrows,matrix,positioning,calc,automata,patterns}
\usepackage{booktabs}
\usepackage{dcolumn}
\usepackage{multirow}
\usepackage{diagbox}
\usepackage{tabularx}
\usepackage{verbatim}
\usepackage{listings}
\usepackage[ruled,vlined]{algorithm2e}
\usepackage{fancyvrb}
\usepackage{hyperref}
\usepackage[round]{natbib}
\usepackage{sectsty}

\hypersetup{
    bookmarks=true,         
    unicode=false,          
    pdftoolbar=true,        
    pdfmenubar=true,        
    pdffitwindow=false,     
    pdfstartview={FitH},    
    pdftitle={My title},    
    pdfauthor={Author},     
    pdfsubject={Subject},   
    pdfcreator={Creator},   
    pdfproducer={Producer}, 
    pdfkeywords={key1, key2}, 
    pdfnewwindow=true,      
    colorlinks=true,        
    linkcolor=blue,         
    citecolor=blue,         
    filecolor=blue,         
    urlcolor=cyan           
}

\usepackage{stackengine}
\stackMath
\newcommand\tenq[2][1]{%
\def\useanchorwidth{T}%
\ifnum#1>1%
\stackunder[0pt]{\tenq[\numexpr#1-1\relax]{#2}}{\!\scriptscriptstyle\thicksim}%
\else%
\stackunder[1pt]{#2}{\!\scriptstyle\thicksim}%
\fi%
}

\makeatletter
\DeclareRobustCommand\widecheck[1]{{\mathpalette\@widecheck{#1}}}
\def\@widecheck#1#2{%
    \setbox\z@\hbox{\m@th$#1#2$}%
    \setbox\tw@\hbox{\m@th$#1%
       \widehat{%
          \vrule\@width\z@\@height\ht\z@
          \vrule\@height\z@\@width\wd\z@}$}%
    \dp\tw@-\ht\z@
    \@tempdima\ht\z@ \advance\@tempdima2\ht\tw@ \divide\@tempdima\thr@@
    \setbox\tw@\hbox{%
       \raise\@tempdima\hbox{\scalebox{1}[-1]{\lower\@tempdima\box
\tw@}}}%
    {\ooalign{\box\tw@ \cr \box\z@}}}
\makeatother

\def\given{\,|\,}

\def\Biggiven{\,\Big{|}\,}
\def\tr{\mathop{\text{tr}}\kern.2ex}
\def\tZ{{\tilde Z}}

\def\P{{\mathrm P}}

\def\E{{\mathrm E}}

\def\d{{\mathrm d}}

\newcommand{\zahl}[1]{\llbracket #1\rrbracket}
\newcommand\yestag{\addtocounter{equation}{1}\tag{\theequation}}

\newcolumntype{L}[1]{>{\raggedright\let\newline\\\arraybackslash\hspace{0pt}}m{#1}}
\newcolumntype{C}[1]{>{  \centering\let\newline\\\arraybackslash\hspace{0pt}}m{#1}}
\newcolumntype{R}[1]{>{ \raggedleft\let\newline\\\arraybackslash\hspace{0pt}}m{#1}}
\newcolumntype{d}[1]{D{.}{.}{#1}}
\newcolumntype{H}{>{\setbox0=\hbox\bgroup}c<{\egroup}@{}}
\newcolumntype{Z}{>{\setbox0=\hbox\bgroup}c<{\egroup}@{\hspace*{-\tabcolsep}}}
\newcolumntype{b}{X}
\newcolumntype{s}{>{\hsize=.5\hsize}X}

\DeclarePairedDelimiter\abs{\lvert}{\rvert}%

\DeclarePairedDelimiter\floor{\lfloor}{\rfloor}

\numberwithin{equation}{section}

\newtheorem{theorem}{Theorem}[section]
\newtheorem{lemma}{Lemma}[section]

\newtheorem{assumption}{Assumption}[section]
\newtheorem{corollary}{Corollary}[section]
\newtheorem{innercustomgeneric}{\customgenericname}
\providecommand{\customgenericname}{}
\newcommand{\newcustomtheorem}[2]{%
  \newenvironment{#1}[1]
  {%
   \renewcommand\customgenericname{#2}%
   \renewcommand\theinnercustomgeneric{##1}%
   \innercustomgeneric
  }
  {\endinnercustomgeneric}
}
\newcustomtheorem{customdefinition}{Definition}
\newcustomtheorem{customdefinitions}{Definitions}
\newcustomtheorem{customtheorem}{Theorem}
\newcustomtheorem{customassumption}{Assumption}
\newcustomtheorem{customlemma}{Lemma}
\newcustomtheorem{customexample}{Example}
\theoremstyle{definition}
\newtheorem{definition}{Definition}[section]

\usepackage{enumitem}
\makeatletter
\newcommand{\mylabel}[2]{#2\def\@currentlabel{#2}\label{#1}}
\makeatother

\graphicspath{{./fig3/}}

\allowdisplaybreaks

\begin{document}

\setlength{\abovedisplayskip}{5pt}
\setlength{\belowdisplayskip}{5pt}
\setlength{\abovedisplayshortskip}{5pt}
\setlength{\belowdisplayshortskip}{5pt}
\hypersetup{colorlinks,breaklinks,urlcolor=blue,linkcolor=blue}

\title{\LARGE On the consistency of bootstrap for matching estimators}

\author{Ziming Lin\thanks{Department of Statistics, University of Washington, Seattle, WA 98195, USA; e-mail: {\tt zmlin@uw.edu}} ~~~and~~~Fang Han\thanks{Department of Statistics, University of Washington, Seattle, WA 98195, USA; e-mail: {\tt fanghan@uw.edu}}
}

\date{\today}

\maketitle

\vspace{-1em}

\begin{abstract}
In a landmark paper, \cite{abadie2008failure} showed that the naive bootstrap is inconsistent when applied to nearest neighbor matching estimators of the average treatment effect with a fixed number of matches. Since then, this finding has inspired numerous efforts to address the inconsistency issue, typically by employing alternative bootstrap methods. In contrast, this paper shows that the naive bootstrap is provably consistent for the original matching estimator, provided that the number of matches, $M$, diverges. The bootstrap inconsistency identified by \cite{abadie2008failure} thus arises solely from the use of a fixed $M$.
\end{abstract}


\section{Introduction}

Nearest neighbor (NN) matching methods, which impute missing potential outcomes by averaging outcomes from matches in the opposite treatment group, are among the most widely used tools for causal inference \citep{stuart2010matching, imbens2024causal}. These methods are praised as “intuitive and easy to explain” \citep{imbens2024causal} and, unlike competing approaches \citep{heckman1997matching, heckman1998matching, hirano2003efficient, chernozhukov2018double}, hinge on an easy-to-interpret parameter—the number of matches—making them user-friendly to tune.

However, these advantages are offset by a significant limitation: Efron’s naive bootstrap \citep{efron1979bootstrap} fails to capture the stochastic behavior of NN matching estimators when the number of matches, $M$, is {\it fixed}. Specifically, the naive bootstrap produces inconsistent estimates of both the asymptotic distribution and the variance of the estimators. As a result, NN matching estimators fall into an interestingly rare class of statistics that are asymptotically normal, $\sqrt{n}$ -consistent, but bootstrap-inconsistent. This phenomenon was rigorously established through a series of seminal papers by Abadie and Imbens \citep{abadie2006large, abadie2008failure, abadie2011bias} and echoed in parallel work on Chatterjee’s rank correlation \citep{chatterjee2020new, lin2022limit, lin2023failure}.

The bootstrap inconsistency of NN matching estimators has spurred extensive research on alternative resampling methods. Proposed solutions include the wild bootstrap, which resamples the linear component of the estimator \citep{otsu2017bootstrap, adusumilli2018bootstrap, bodory2024nonparametric}, the block bootstrap, which resamples matched sets instead of individual observations \citep{abadie2022robust}, the $m$-out-of-$n$ bootstrap \citep{walsh2023nearest}, and the bag of little bootstraps \citep{kosko2024fast}, among others. Some empirical evidence suggests that bootstrap-based methods improve coverage rates over asymptotic approximations; see, for example, \cite{bodory2020finite} and references therein.

This paper contributes to this line of research by revisiting the original setting of \cite{abadie2008failure}. Building on recent insights from \cite{lin2023estimation} and \cite{he2024propensity}, which connect NN matching to augmented inverse probability weighting estimators and highlight the importance of allowing $M\to\infty$ for improved efficiency, we show that the counterexample in \cite{abadie2008failure} no longer holds when $M$ grows with the sample size. Specifically, we demonstrate that once $M$ is allowed to diverge, the naive bootstrap yields consistent distributional estimates for NN matching estimators. Thus, the bootstrap inconsistency identified in \cite{abadie2008failure} arises entirely from the practical constraint of fixing $M$.

\section{Setup and methods}

\subsection{Matching methods}\label{sec:matching}

We consider the standard potential outcome causal model for a binary treatment, assuming $n$  realizations, $\{(X_i, D_i, Y_i(0), Y_i(1))\}_{i=1}^n$, drawn from a quadruple 
\[
(X, D, Y(0), Y(1)).
\]
Here $Y(0),Y(1) \in \bR$ are the potential outcomes under control and treatment, respectively \citep{neyman1923applications, rubin1974estimating}, the treatment indicator $D \in \{0,1\}$ denotes whether the subject received the treatment, and $X \in \cX \subset \bR^d$ corresponds to pretreatment covariates. For each subject $i\in\zahl{n}:=\{1,2,\ldots,n\}$, we observe only the triple 
\[
O_i := (X_i, D_i, Y_i = Y_i(D_i)). 
\]
Write $\mX=[X_1,\ldots,X_n]\in\bR^{d\times n}$, $\mD=(D_1,\ldots, D_n)^\top\in\bR^n$, and $\mY=(Y_1,\ldots,Y_n)^\top\in\bR^n$. The goal of this paper is to estimate the population average treatment effect (ATE),
\begin{align*}
    \tau := \E[Y(1) - Y(0)],
\end{align*}
using only the observed data $(\mX,\mD,\mY)$.

Abadie and Imbens \citep{abadie2006large, abadie2011bias} introduced a (bias-corrected) matching approach to estimating $\tau$. Their method imputes the missing potential outcome $Y_i(1-D_i)$ for each subject $i$ by averaging the outcomes $Y_j(D_j)$'s from subjects with opposite treatment status, i.e., those whose covariates are closest to $X_i$ under the Euclidean distance $\|\cdot\|$ with $D_j=1-D_i$. The resulting estimator is then adjusted by subtracting a bias estimate.

In detail, for each subject $i\in\zahl{n}$, let $\mathcal{J}_M(i)$ denote the index set of $M$-NNs of $X_i$ among subjects with $D_j=1-D_i$; formally,
\[
\mathcal{J}_M(i) := \Big\{ j\in\zahl{n}: D_j=1-D_i, \sum_{\ell=1, D_\ell =1- D_i}^n \ind \Big( \lVert X_\ell - X_i \rVert < \lVert X_j - X_i \rVert\Big) < M\Big\}.
\]
The missing potential outcome $Y_i(1-D_i)$ is then imputed by
\[
\hat Y_i(1-D_i) := \frac{1}{M}\sum_{j\in\cJ_M(i)}Y_j,
\]
along with the convention that $\hat Y_i(D_i)=Y_i$. The {\it bias-uncorrected} $M$-NN matching estimator of $\tau$ is accordingly given by
\[
\hat\tau_M := \frac{1}{n}\sum_{i=1}^n\Big\{\hat Y_i(1)-\hat Y_i(0)  \Big\},
\]
as studied in \cite{abadie2006large}. 

However, as noted in \cite{abadie2006large}, unless $d=1$, the estimator $\hat\tau_M$ suffers from an asymptotically non-negligible bias of the form
\[
B_M :=  \frac{1}{n}\sum_{i = 1}^n\frac{2D_i - 1}{M}\sum_{j \in \cJ_M(i)}\Big\{\mu_{1 - D_i}(X_i) - \mu_{1- D_i}(X_j)\Big\},
\]
where $\mu_0(\cdot)$ and $\mu_1(\cdot)$ are the outcome conditional means 
\[
\mu_0(x) := \E[Y \given X=x, D=0]~~ {\rm and} ~~\mu_1(x) := \E[Y \given X=x, D=1],
\] 
respectively. To address this bias, \cite{abadie2011bias} proposed a simple bias correction. Let $\hat \mu_0(\cdot)$ and $\hat \mu_1(\cdot)$ be estimators of $\mu_0(\cdot)$ and $\mu_1(\cdot)$, respectively. The bias correction term is defined as
\[
\hat B_M := \frac{1}{n}\sum_{i=1}^n\frac{2D_i-1}{M}\sum_{j\in\cJ_M(i)}\Big\{\hat\mu_{1-D_i}(X_i)-\hat\mu_{1-D_i}(X_j)\Big\}.
\]
The final bias-corrected $M$-NN matching estimator is
\[
\hat\tau_M^{\rm bc}:=\hat\tau_M-\hat B_M.
\]

\subsection{The naive bootstrap}\label{sec:bootstrap}

This paper focuses on Efron’s naive bootstrap for inferring $\hat\tau_M^{\rm bc}$. Specifically, we consider the bootstrap sample defined by \cite{abadie2008failure},
\[
\Big\{O_i^*=(X_i^*,D_i^*,Y_i^*)\Big\}_{i=1}^n,
\]
where observations are independently resampled {\it within} the treated and control groups separately. The bootstrap sample thus consists of $n_0$ and $n_1$ independent draws with replacement from $\{O_i\}_{i=1,D_i=0}^n$ and $\{O_i\}_{i=1,D_i=1}^n$, respectively.

Next, we implement the bootstrap procedure to infer the matching estimator. Following the same steps as in Section \ref{sec:matching}, we first define the bootstrap matching set as
\[
\tilde{\mathcal{J}}_M^*(i) := \Big\{ j\in\zahl{n}: D_j^*=1-D_i^*, \sum_{\ell=1, D_\ell^* =1- D_i^*}^n \ind \Big( \lVert X_\ell^* - X_i^* \rVert < \lVert X_j^* - X_i^* \rVert\Big) < M\Big\}.
\] 
Since ties can occur in the bootstrap sample,  $\tilde{\mathcal{J}}_M^*(i)$ may contain more than $M$ elements. Similar to \cite{lin2023failure}, we break ties {\it arbitrarily} to obtain a subset 
\[
\mathcal{J}_M^*(i)\subset \tilde{\mathcal{J}}_M^*(i),~~~\text{with cardinality }~\Big|\mathcal{J}_M^*(i)\Big|=M
\]
for each $i\in\zahl{n}$. The corresponding imputations are then 
\begin{align*}
\hat{Y}_i^*(0) =
    \begin{cases}
         M^{-1} \sum_{j \in \mathcal{J}_M^*(i)} Y_j^*,  & \mbox{ if } D_i^*=1, \\   
         Y_i^*, & \mbox{ if } D_i^*=0,
         \end{cases}
~~{\rm and}~~
\hat{Y}_i^*(1) =
    \begin{cases}
        Y_i^*, & \mbox{ if } D_i^*=1, \\   
       M^{-1}\sum_{j \in \mathcal{J}_M^*(i)}Y_j^*, & \mbox{ if } D_i^*=0.
    \end{cases}    
\end{align*}
The bias-uncorrected bootstrap matching estimator is accordingly given by
\[
\hat\tau^*_M:=\frac{1}{n}\sum_{i=1}^n\Big\{\hat Y_i^*(1)-\hat Y_i^*(0)\Big\},
\]
and the bias-corrected bootstrap matching estimator is 
\[
\hat\tau_M^{*, {\rm bc}} := \hat\tau^*_M-\hat B_M^*,
\]
where
\[
\hat B_M^* := \frac{1}{n}\sum_{i=1}^n\frac{2D_i^*-1}{M}\sum_{j\in\cJ^*_M(i)}\Big\{\hat\mu_{1-D_i^*}(X_i^*)-\hat\mu_{1-D_i^*}(X_j^*)\Big\}.
\]

\section{Bootstrap theory}

This section gives the bootstrap consistency guarantee for $\hat\tau_M^{*, {\rm bc}}$ provided that $M$ grows to infinity at an appropriate rate with the sample size. 

\subsection{Setup}

Let us first rigorously document the two sources of randomness: one from the original data generating scheme and the other from the bootstrap. Efron's bootstrap corresponds to some random weights $\mW_0$ and $\mW_1$ that account for the numbers of times each observation in $\{O_i\}_{i=1,D_i=1}^n$ and $\{O_i\}_{i=1,D_i=0}^n$ appears in the bootstrap sample. Specifically, we have 
\begin{align*}
   \mW_0 \sim {\rm Mult}\big(n_0, (n_0^{-1}, \ldots, n_0^{-1})\big)  ~\text{ and }~ \mW_1 \sim {\rm Mult}\big(n_1, (n_1^{-1}, \ldots, n_1^{-1})\big)
\end{align*}
to be multinomially distributed. Let $\mW=(W_1,\ldots,W_n)^\top$ be the concatenated vector of $\mW_0$ and $\mW_1$ according to the order of $O_i$'s. 

In the following, we use $(\Omega_O, \cA_O, \P_O )$ to denote the probability space from which the observed data $\{O_i\}_{i=1}^n$ arises and $( \Omega_W, \cA_W, \P_W )$ to denote the probability space for the bootstrap weights. Let the probability measure on the product measurable space be $\P_{OW}$,
so that 
\begin{align*}
    \big( \Omega_O^\infty, \cA_O^\infty, \P_O^\infty \big) \times ( \Omega_W, \cA_W, \P_W ) =  \big( \Omega_O^\infty \times \Omega_W, \cA_O^\infty \times \cA_W, \P_{OW}\big),
    \yestag \label{eq:setup,ProbSpace}
\end{align*}
which we call the {\it bootstrap space}; throughout this paper, we make the assumption that the bootstrap weights are independent of the data, so that $\P_{OW} = \P_O^\infty \times \P_W$. Denote the corresponding expectation with respect to $\P_{OW}$ by $\E_{OW}$, similarly for the notations $\E_{W|O}$, $\E_{O}$, and $\E_{W}$. For simplicity, if no confusion is possible, we shorthand $\P_O^\infty, \P_W$ as $\P$, and $\E_{O}, \E_{W}$ as $\E$.

Lastly, given a real-valued random quantity $R_n$ defined over the bootstrap space, we say $R_n$ is of an order $o_{\P_W}(1)$ in $\P_O$-probability if, for any $\epsilon, \delta>0$
\begin{align*}
    \P_O \Big\{ \P_{W|O}(\lvert R_n \rvert > \epsilon) > \delta \Big\} \to 0, ~~~~~\text{as } n \to \infty.
\end{align*}

\subsection{Theory}

In the following, for any two sequences of positive real numbers $\{a_n\}$ and $\{b_n\}$, we write $a_n\lesssim b_n$, or equivalently $a_n=O(b_n)$, if there exists a universal constant $C>0$ such that $a_n/b_n\leq C$ for all sufficiently large $n$. We write $a_n=o(b_n)$ if $a_n/b_n\to 0$ as $n\to\infty$. We use $o_\P(\cdot)$ and $O_\P(\cdot)$ to represent the stochastic versions of $o(\cdot)$ and $O(\cdot)$, respectively.

Our theory is based on the following two sets of assumptions.

\begin{assumption}  \phantomsection \label{asp:matching1} 
  \begin{enumerate}[itemsep=-.5ex,label=(\roman*)]
    \item For almost all $x \in \cX$, $D$ is independent of $(Y(0),Y(1))$ conditional on $X=x$, and there exists some constant $\eta > 0$ such that 
    \[
    \eta < e(x):=\P(D=1 \given X=x) < 1-\eta.
    \]
    \item $\{(X_i,D_i,Y_i(0),Y_i(1))\}_{i=1}^n$ are independent copies of $(X,D,Y(0),Y(1))$.
    \item $\sigma_\omega^2(x):=\E [U^2_\omega \given X=x] =\E[(Y(\omega)-\mu_{\omega}(X))^2\given X=x]$ is uniformly bounded by a constant $C_{\sigma}>0$ for almost all $x \in \cX$ and $\omega \in \{0,1\}$.
    \item $\E [\mu^2_\omega]$ is bounded for $\omega \in \{0,1\}$.
  \end{enumerate}
\end{assumption}

\begin{assumption}  \phantomsection \label{asp:matching2} 
  \begin{enumerate}[itemsep=-.5ex,label=(\roman*)]    
   \item For any $\omega \in \{0,1\}$,  $\max_{t \in \Lambda_{\floor{d/2}{}+1}} \lVert \partial^t {\mu}_\omega \rVert_{\infty}$ is bounded, where $\|\cdot\|_{\infty}$ stands for the function supremum norm, $\Lambda_k$ denotes the set of all $d$-dimension vectors of nonnegative integers that sums to $k$, and $\floor{\cdot}{}$ stands for the floor function.  
   \item For any $\omega \in \{0,1\}$, the outcome regression estimator $\hat{\mu}_\omega(\cdot)$ satisfies
    \begin{align*}
    \max_{t \in \Lambda_{\floor{d/2}{}+1}} \big\lVert \partial^t \hat{\mu}_\omega \big\rVert_{\infty} = O_{\P_O}(1) 
    \end{align*}
    and 
    \begin{align*}
        \max_{t \in \Lambda_{\ell}} \big\lVert \partial^t \hat{\mu}_\omega - \partial^t {\mu}_\omega \big\rVert_{\infty}= O_{\P_O} ( n^{-\gamma_{\ell}}), ~~\text{for all } \ell \in \zahl{\floor{d/2}{}}
    \end{align*}
    with some constants $\{\gamma_\ell \}_{\ell=1}^{\floor{d/2}{}}$ satisfying $\gamma_\ell > \frac{1}{2} - \frac{\ell}{d}$ for all $\ell \in \zahl{\floor{d/2}{}}$.
      \end{enumerate}
\end{assumption}

Assumption \ref{asp:matching1} corresponds to Assumption 4.1 in \cite{lin2023estimation}, and is used for proving the asymptotic normality of $\hat\tau_M$. Assumption \ref{asp:matching2} is needed for justifying bias correction. It aligns with Assumptions 4.4–4.5 in \cite{lin2023estimation}, except that Assumption 4.4(i)(ii) has been omitted. This relaxation is made possible by a new linearization result, which we will present in Theorem \ref{thm-new} ahead. It eliminates the need for a Lindeberg-Feller-type argument used in the proof of Lemma C.1 in \cite{lin2023estimation}. 

With the above two assumptions, the following theorem is the main result in \cite{lin2023estimation}.

\begin{theorem}[\cite{lin2023estimation}]\label{thm:lin} Assume that Assumption \ref{asp:matching1} holds true and either $(\P_{X\given D=0}, \P_{X\given D=1})$ or  $(\P_{X\given D=1}, \P_{X\given D=0})$ satisfies Assumption~\ref{asp:BT,Lp,risk} ahead. 
\begin{enumerate}[itemsep=-.5ex,label=(\roman*)]
\item If $M \to \infty$ and $M \log n/n \to 0$ as $n \to \infty$, we then have
\[
\sup_{t \in \bR} \Big\lvert \P\big\{\sqrt{n}(\hat\tau_M-B_M- \tau)\le t\big\} -\P\big(N(0, \sigma^2) \le t\big)   \Big\rvert = o(1),
\]
where 
\[
\sigma^2:=\E\Big[(\mu_1(X)-\mu_0(X)-\tau)^2 \Big]+\E\Big[\frac{\sigma_1^2(X)}{e(X)}+\frac{\sigma_0^2(X)}{1-e(X)}\Big]
\]
is the semiparametric efficiency lower bound \citep{hahn1998role}.
\item If further assuming Assumption \ref{asp:matching2} and $M/n^\gamma\to 0$ with
\[
  \gamma := \Big\{ \min_{\ell \in \zahl{\lfloor d/2 \rfloor} } \Big[1-\Big(\frac{1}{2} - \gamma_\ell \Big)\frac{d}{\ell}\Big]\Big\} \wedge \Big[1-  \frac{d/2}{\lfloor d/2 \rfloor + 1}\Big],
  \]
 we then have
\[
\sup_{t \in \bR} \Big\lvert \P\big(\sqrt{n}(\hat\tau_M^{\rm bc}- \tau)\le t\big) -\P\big(N(0, \sigma^2) \le t\big)   \Big\rvert = o(1).
\]
\end{enumerate}
\end{theorem}

Our first contribution to the literature is a new linear representation result for the matching estimator $\hat\tau_M$.

\begin{theorem}[Linear representation for $\hat\tau_M$]\label{thm-new} 
 Assume Assumptions~\ref{asp:matching1} and either $(\P_{X\given D=0}, \P_{X\given D=1})$ or  $(\P_{X\given D=1}, \P_{X\given D=0})$ satisfies Assumption~\ref{asp:BT,Lp,risk} ahead. If $M \to \infty$ and $M \log n/n \to 0$ as $n \to \infty$, we then have
 \[
 \sqrt{n} (\hat\tau_M-B_M-\tau) = \frac{1}{\sqrt{n}} \sum_{i=1}^n \big\{\chi_0(O_i)-\tau\big\} + o_{\P_O}(1), 
 \]
 where 
  \begin{align*}
      \chi_0(O_i) := \mu_1(X_i) - \mu_0(X_i) + \Big\{\frac{D_i}{e(X_i)} - \frac{1-D_i}{1-e(X_i)}\Big\}\big\{Y_i - \mu_{D_i}(X_i) \big\}.
  \end{align*}
\end{theorem}

Building on the insights from \cite{lin2023estimation}, Theorem \ref{thm-new} offers a new perspective on matching estimators. Unlike the proof of Lemma C.1 in \cite{lin2023estimation}, which establishes a central limit theorem for $\hat\tau_M$  using an implicit conditioning argument, Theorem \ref{thm-new} provides a straightforward asymptotically linear representation for $\hat\tau_M$. This representation simplifies the subsequent analyses in both the original and bootstrap spaces. Moreover, it connects to the linear expansion initially discovered by \cite{abadie2006large} and later employed by \cite{otsu2017bootstrap}. The differences between our form and theirs, on the other hand, are self-explanatory.

With Theorem \ref{thm-new}, we are now ready to introduce the main theorem of this paper.

\begin{theorem}[Bootstrap distribution consistency]\label{thm2} 
 Assume that Assumption \ref{asp:matching1} holds true and either $(\P_{X\given D=0}, \P_{X\given D=1})$ or  $(\P_{X\given D=1}, \P_{X\given D=0})$ satisfies Assumption~\ref{asp:BT,Lp,risk} ahead. If $M \to \infty$, $M \log n/n \to 0$, and $\log(n/M)/M \to 0$ as $n \to \infty$, we then have
    \begin{enumerate}[itemsep=-.5ex,label=(\roman*)]
      \item\label{thm2:1} {\rm the linearization of $\hat\tau_M^*$:} 
      \[
      \sqrt{n} (\hat \tau_M^* - B_M^*-\tau) = \frac{1}{\sqrt{n}} \sum_{i=1}^n W_i \big\{\chi_0(O_i)-\tau\big\} + o_{\P_{W}}(1) \text{ in }\P_O\text{-probability}, 
      \]
      where
      \[
      B_M^* :=  \frac{1}{n}\sum_{i = 1}^n\frac{2D_i^* - 1}{M}\sum_{j \in \cJ_M^*(i)}\Big\{\mu_{1 - D_i^*}(X_i^*) - \mu_{1- D_i^*}(X_j^*)\Big\};
      \]
      \item\label{thm2:2} further assuming  Assumption \ref{asp:matching2} and $M/n^\gamma\to 0$, 
      \[
      \sqrt{n} (\hat \tau_M^{*, {\rm bc}}-\hat\tau_M^{\rm bc}) = \frac{1}{\sqrt{n}} \sum_{i=1}^n (W_i-1) \chi_0(O_i) + o_{\P_{W}}(1) \text{ in }\P_O\text{-probability,} \yestag \label{eq:matching,BTGn}
      \]
so that 
    \[
    \sup_{t \in \bR} \Big\lvert \P_{\mW|\mO}\big(\sqrt{n} (\hat\tau_M^{*,{\rm bc}} - \hat\tau_M^{\rm bc}) \le t \big) - \P\big(N(0, \sigma^2) \le t\big) \Big\rvert = o_{\P_O}(1)
    \]
and
    \[
    \sup_{t \in \bR} \Big\lvert \P_{\mW|\mO}\big(\sqrt{n} (\hat\tau_M^{*, {\rm bc}} - \hat\tau_M^{\rm bc}) \le t \big) - \P\big(\sqrt{n}(\hat\tau_M^{\rm bc}- \tau) \le t\big) \Big\rvert = o_{\P_O}(1). 
    \]
        \end{enumerate}
\end{theorem}

A key component of the proof of Theorem \ref{thm2} is demonstrating that the NN-matching-based density ratio estimator, discussed in \citet[Section 2]{lin2023estimation}, remains a consistent estimator of the density ratio in the bootstrap space. This result, which will be presented in Section \ref{sec:dr} ahead, represents one of the main technical contributions of this paper. Of note, the bootstrap validity is obtained under nearly the same conditions as Theorem \ref{thm:lin}, which slightly weakens the requirement of \citet[Theorem 4.1(ii)]{lin2023estimation}.

In the special case of $d=1$, \cite{abadie2006large} noted that there is no need of bias correction and the above theorem could be further simplified. The following corollary shows that their note is also correct when $M$ is allowed to vary with the sample size.

\begin{corollary}[Bootstrap distribution consistency, the special case of $d=1$]\label{coro1}  Assume that Assumption \ref{asp:matching1} holds true and either $(\P_{X\given D=0}, \P_{X\given D=1})$ or  $(\P_{X\given D=1}, \P_{X\given D=0})$ satisfies Assumption~\ref{asp:BT,Lp,risk} ahead. Further assume that $\mu_0(x)$ and $\mu_1(x)$ are Lipschitz on $\cX$ and $d=1$. Then, if $M \to \infty$, $M^2/n \to 0$, and $\log(n/M)/M \to 0$ as $n \to \infty$, we have
    \begin{enumerate}[itemsep=-.5ex,label=(\roman*)]
    \item $\sqrt{n} (\hat\tau_M-\tau) = \frac{1}{\sqrt{n}} \sum_{i=1}^n (\chi_0(O_i)-\tau) + o_{\P_O}(1)$;
      \item $\sqrt{n} (\hat \tau_M^* -\tau) = \frac{1}{\sqrt{n}} \sum_{i=1}^n W_i(\chi_0(O_i)-\tau) + o_{\P_{W}}(1)$  in $\P_O$-probability;
      \item $
      \sqrt{n} (\hat \tau_M^{*}-\hat\tau_M) = \frac{1}{\sqrt{n}} \sum_{i=1}^n (W_i-1) \chi_0(O_i) + o_{\P_{W}}(1)$
in $\P_O$-probability, so that
    \[
    \sup_{t \in \bR} \Big\lvert \P_{\mW|\mO}\big(\sqrt{n} (\hat\tau_M^{*} - \hat\tau_M) \le t \big) - \P\big(N(0, \sigma^2) \le t\big) \Big\rvert = o_{\P_O}(1)
    \]
and
    \[
    \sup_{t \in \bR} \Big\lvert \P_{\mW|\mO}\big(\sqrt{n} (\hat\tau_M^{*} - \hat\tau_M) \le t \big) - \P\big(\sqrt{n}(\hat\tau_M^{\rm bc}- \tau) \le t\big) \Big\rvert = o_{\P_O}(1). 
    \]
        \end{enumerate}
\end{corollary}

\section{Density ratio estimation in the bootstrap space}\label{sec:dr}

Like \cite{lin2023estimation}, for giving the proof of bootstrap consistency, we have to first discuss a bootstrap counterpart of NN matching-based density ratio estimator that receives similar attention in \citet[Section 2]{lin2023estimation}. 

With a slight abuse of notation, consider $X,Z$ to be two general random vectors in $\bR^d$ that are defined on the same probability space with laws $\nu_0$ and $\nu_1$, respectively. Furthermore, assume that both $\nu_0$ and $\nu_1$ are absolutely continuous with respect to the Lebesgue measure $\lambda$ on $\bR^d$ that is equipped with the Euclidean norm $\lVert \cdot \rVert$. Denote the Lebesgue densities of  $\nu_0$ and $\nu_1$ by $f_0$ and $f_1$, respectively. Lastly, assume that $\nu_1$ is absolutely continuous with respect to $\nu_0$ and denote the corresponding density ratio by 
\[
r(x) := f_1(x)/f_0(x),~ \text{ with the convention that } 0/0=0.
\]

Suppose we have the original sample comprised of $N_0$ copies $\{X_i\}_{i=1}^{N_0}$ of $X$ and $N_1$ copies $\{Z_i\}_{i=1}^{N_1}$ of $Z$, and assume that they are mutually independent. Note that, by the setup, there is almost surely no ties in the observations. 

\begin{definition}[NN matching, Definition 2.1 in \cite{lin2023estimation}]\label{def:NNmatching}
 For any $x, z \in \bR^d$ and $M \in \zahl{N_0}$, 
\begin{enumerate}[itemsep=-.5ex,label=(\roman*)]
 \item \label{def:NN-matching} let $\cX_{(M)}(\cdot): \bR^d \to \{X_i\}_{i=1}^{N_0}$ 
  be the map 
  that returns the value of the input $z$'s $M$-th NN in the original sample $\{X_i\}_{i=1}^{N_0}$, i.e., the value of $x \in \{X_i\}_{i=1}^{N_0}$ such that
  \begin{align*}
    \sum_{i=1}^{N_0} \ind\big(\lVert X_i - z \rVert \le \lVert x - z \rVert\big) = M;
  \end{align*}
  \item\label{def:NN-matching-2}  let $K_M(\cdot): \bR^d\to \zahl{N_1}$ be the map  that returns the number of matched times of $x$, i.e., 
  \begin{align*}
    K_M(x) = K_M\big(x;\{X_i\}_{i=1}^{N_0},\{Z_j\}_{j=1}^{N_1}\big) := \sum_{j=1}^{N_1} \ind\big(\lVert x - Z_j \rVert \le \lVert \cX_{(M)}(Z_j) - Z_j \rVert\big);
  \end{align*}
  \item\label{def:NN-matching-3} let $A_M(\cdot): \bR^d\to \cB(\bR^d)$ be the corresponding map from $\bR^d$ to the class of all Borel sets in $\bR^d$ so that
    \begin{align*}
    A_M(x) = A_M\big(x;\{X_i\}_{i=1}^{N_0}\big) := \big\{z \in \bR^d : \lVert x - z \rVert \le \lVert \cX_{(M)}(z) - z \rVert\big\}
  \end{align*}
   returns the catchment area of $x$ in the setting of Definition \ref{def:NNmatching}\ref{def:NN-matching-2}; 
\end{enumerate}
\end{definition}
Now, suppose we resample $\{X_i\}_{i=1}^{N_0}$ with replacement $N_0$ times according to the naive bootstrap weights $\mW_X \sim {\rm Mult}(N_0;N_0^{-1}, \ldots, N_0^{-1})$. Here each $W_{X_i}$ refers to the number of times that $X_i$ is sampled. Similarly, we resample $\{Z_i\}_{i=1}^{N_1}$ with replacement $N_1$ times according to the naive bootstrap weights $\mW_Z \sim {\rm Mult}(N_1;N_1^{-1}, \ldots, N_1^{-1})$, and let $W_{Z_j}$ refer to the number of times that $Z_j$ is sampled. 

The bootstrap sample then yields the following counterparts of Definition \ref{def:NNmatching}.

\begin{definition}[bootstrap NN matching]
 For any $x, z \in \bR^d$ and $M \in \zahl{N_0}$, 
\begin{enumerate}[itemsep=-.5ex,label=(\roman*)]
 \item\label{def:BT-NN-matching} let $\cX_{(M)}^*(\cdot): \bR^d \to \{X_i^*\}_{i=1}^{N_0}$ 
  be the map 
  that returns the value of the input $z$'s $M$-th NN in the bootstrap sample $\{X_i^*\}_{i=1}^{N_0}$, i.e., the value of $x \in \{X_i^*\}_{i=1}^{N_0}$ such that
  \begin{align*}
    \sum_{i=1}^{N_0} \ind\big(\lVert X_i^* - z \rVert < \lVert x - z \rVert\big) < M ~\text{ and }~ \sum_{i=1}^{N_0} \ind\big(\lVert X_i^* - z \rVert \le \lVert x - z \rVert\big) \ge M;
  \end{align*}
\item  let $K_M^{*,+}(\cdot), K_M^{*,-}(\cdot): \bR^d\to \zahl{N_1}$ be the two bootstrap matched times of $x$ that account for ties in different ways, i.e.,
\begin{align*}
    K_M^{*,+}(x) &= K_M^{*,+}\big(x;\{X_i^*\}_{i=1}^{N_0},\{Z_j^*\}_{j=1}^{N_1}\big) := \sum_{j=1}^{N_1} \ind\big(\lVert x-Z_j^*\rVert \le \lVert \cX_{M}^*(Z_j^*)-Z_j^*\rVert\big), \\
    K_M^{*,-}(x) &= K_M^{*,-}\big(x;\{X_i^*\}_{i=1}^{N_0},\{Z_j^*\}_{j=1}^{N_1}\big) := \sum_{j=1}^{N_1} \ind\big(\lVert x-Z_j^*\rVert < \lVert \cX_{M}^*(Z_j^*)-Z_j^*\rVert\big);
\end{align*}
  \item\label{def:BT-NN-matching-3} let $A_M^{*,+}(\cdot), A_M^{*,-}(\cdot): \bR^d\to \cB(\bR^d)$ be the corresponding maps from $\bR^d$ to the class of all Borel sets in $\bR^d$ so that
    \begin{align*}
    A_M^{*,+}(x) &= A_M^{*,+}\big(x;\{X_i^*\}_{i=1}^{N_0}\big) := \big\{z \in \bR^d : \lVert x - z \rVert \le \lVert \cX_{(M)}^*(z) - z \rVert\big\} \\
{\rm and}~~~    A_M^{*,-}(x) &= A_M^{*,-}\big(x;\{X_i^*\}_{i=1}^{N_0}\big) := \big\{z \in \bR^d : \lVert x - z \rVert < \lVert \cX_{(M)}^*(z) - z \rVert\big\}
  \end{align*}
   return the catchment areas of $x$.
\end{enumerate}
\end{definition}

We then give the bootstrap density ratio estimators, $\hat{r}_M^{*,+}(x), \hat{r}_M^{*,-}(x)$, as
\begin{align*}
    \hat{r}_M^{*,+}\big(x\big) &= \hat{r}_M^{*,+}\big(x; \{X_i^*\}_{i=1}^{N_0},\{Z_j^*\}_{j=1}^{N_1}\big) := \frac{N_0}{N_1} \frac{K_M^{*,+}(x)}{M} \\
{\rm and}~~~  \hat{r}_M^{*,-}\big(x\big) &= \hat{r}_M^{*,-}\big(x; \{X_i^*\}_{i=1}^{N_0},\{Z_j^*\}_{j=1}^{N_1}\big) := \frac{N_0}{N_1} \frac{K_M^{*,-}(x)}{M}.
\end{align*}
Note  the following relation that is immediate from the above definitions:
\begin{align*}
    K_M^{*,+}(x) &= \sum_{j=1}^{N_1} \ind\big(Z_j^* \in A_M^{*,+}(x)\big) = \sum_{j=1}^{N_1} W_{Zj} \ind\big(Z_j \in A_M^{*,+}(x)\big) \\
{\rm and}~~~    K_M^{*,-}(x) &= \sum_{j=1}^{N_1} \ind\big(Z_j^* \in A_M^{*,-}(x)\big) = \sum_{j=1}^{N_1} W_{Zj} \ind\big(Z_j \in A_M^{*,-}(x)\big).
\end{align*}

For any fixed point $x\in\bR^d$, we first obtain the following lemma that characterizes the bootstrap catchment areas $A_M^{*,+}(x)$ and $A_M^{*,-}(x)$. This is analogous to \citet[Lemma B.1]{lin2023estimation}.

\begin{lemma}
[Asymptotic $L^p$ moments of bootstrap catchment areas's $\nu_1$-measure]
\label{lemma:moment,Bootstrap catch}
Assuming $M\log N_0/N_0 \to 0$ as $N_0 \to \infty$, then 
\begin{align*}
    \lim_{N_0\to\infty} \frac{N_0}{M} \E_{OW}\big[\nu_1\big(A_M^{*,+}(x)\big)\big] = \lim_{N_0\to\infty} \frac{N_0}{M} \E_{OW}\big[\nu_1\big(A_M^{*,-}(x)\big)\big] = r(x) \yestag \label{eq:BTDR,local,L1}
\end{align*}
holds for $\nu_0$-almost all $x$. If further assuming that $M \to \infty$ and $\log (N_0/M)/M \to 0$ as $N_0 \to \infty$, then for any positive integer $p$, 
\begin{align*}
    \lim_{N_0\to\infty} \Big(\frac{N_0}{M}\Big)^p \E_{OW}\big[\nu_1^p\big(A_M^{*,+}(x)\big)\big] = \lim_{N_0\to\infty} \Big(\frac{N_0}{M}\Big)^p \E_{OW}\big[\nu_1^p\big(A_M^{*,-}(x)\big)\big] = \big[r(x)\big]^p \yestag \label{eq:BTDR,local,Lp}
\end{align*}
  holds for $\nu_0$-almost all $x$.
\end{lemma}

Next, we give a result that is analogous to \citet[Theorem B.1]{lin2023estimation}.

\begin{theorem}[Pointwise consistency in the bootstrap space]\label{thm:cons,BTlp} Assume $M\log N_0/N_0 \to 0$.
\begin{enumerate}[itemsep=-.5ex,label=(\roman*)]
\item\label{thm:cons,BTlp1} (Asymptotic unbiasedness) For $\nu_0$-almost all $x$, we have 
\begin{align*}
    \lim_{N_0\to\infty} \E_{OW} \big[\hat{r}_M^{*,+}(x)\big] =  \lim_{N_0\to\infty} \E_{OW} \big[\hat{r}_M^{*,-}(x)\big] = r(x). \yestag \label{eq:DRest,L1}
\end{align*}
\item\label{thm:cons,BTlp2} (Pointwise $L^p$ consistency) Let $p$ be any positive integer and assume $MN_1/N_0 \to \infty$, $M \to \infty$, and $\log (N_0/M)/M \to 0$ as $N_0 \to \infty$. Then for $\nu_0$-almost all $x$, we have
\begin{align*}
     \lim_{N_0\to\infty} \E_{OW} \big[ \lvert \hat{r}_M^{*,+}(x) - r(x) \rvert^p \big] = \lim_{N_0\to\infty} \E_{OW} \big[ \lvert \hat{r}_M^{*,-}(x) - r(x) \rvert^p \big] = 0. \yestag \label{eq:DRest,Lp}
\end{align*}
\end{enumerate}
\end{theorem}

In the following, let $S_0$ and $S_1$ be the supports of $\nu_0$ and $\nu_1$, respectively. In addition, for any set $S\subset\bR^d$, let ${\rm diam}(S)$ denote its diameter in $(\bR^d,\|\cdot\|)$, i.e., 
\[
{\rm diam}(S):=\sup_{x,y\in S}\|x-y\|.
\]
We then obtain an $L^p$ risk consistency result in the bootstrap space under the following assumption. It is identical to Assumption B.1 in \cite{lin2023estimation}.

\begin{assumption} \phantomsection \label{asp:BT,Lp,risk}
    Assume the following hold.
  \begin{enumerate}[itemsep=-.5ex,label=(\roman*)]
    \item\label{asp:BT,Lp,risk-1} $\nu_0, \nu_1$ are two probability measures on $\bR^d$, both are absolutely continuous with respect to $\lambda$, and $\nu_1$ is absolutely continuous with respect to $\nu_0$.
    \item\label{asp:BT,Lp,risk-2} There exists a constant $R>0$ such that ${\rm diam}(S_0) \le R$.
    \item\label{asp:BT,Lp,risk-3} There exist two constants $f_L,f_U>0$ such that for any $x \in S_0$ and $z \in S_1$, $f_L \le f_0(x) \le f_U$ and $f_1(z) \le f_U$.
    \item\label{asp:BT,Lp,risk-4} There exists a constant $a \in (0,1)$ such that for any $\delta \in (0,{\rm diam}(S_0)]$ and $z \in S_1$,
    \[
      \lambda(B_{z,\delta} \cap S_0) \ge a \lambda(B_{z,\delta}),
    \]
    where $B_{z,\delta}$ represents the closed ball in $\bR^d$ with center at $z$ and radius $\delta$.
  \end{enumerate}
\end{assumption}

The following is the main result in this section; it is analogous to \citet[Theorem B.2]{lin2023estimation} and is the key to the proof of Theorem \ref{thm2}.

\begin{theorem}[$L^p$ risk consistency in the bootstrap space]\label{thm:BT,risk,lp}
Assume the pair of $\nu_0, \nu_1$ satisfies Assumption~\ref{asp:BT,Lp,risk}. Let $p$ be any positive integer. Assume further that $M \to \infty$, $M\log N_0/N_0 \to 0$, $MN_1/N_0 \to \infty$, and $\log (N_0/M)/M \to 0$ as $N_0 \to \infty$. We then have
  \[
    \lim_{N_0\to\infty} \E_{OW} \Big[ \int_{\bR^d} \Big\lvert \hat{r}_M^{*,+}(x) - r(x) \Big\rvert^p f_0(x) \d x \Big] = \lim_{N_0\to\infty} \E_{OW} \Big[ \int_{\bR^d} \Big\lvert \hat{r}_M^{*,-}(x) - r(x) \Big\rvert^p f_0(x) \d x \Big] = 0.
  \]
\end{theorem}

\section{Proofs of main results}\label{sec:proofs}

In  the following, we shorthand $K_M(X_i)$ as $K_M(i)$, $K_M^{*,+}(X_i)$  as $K_M^{*,+}(i)$, and $K_M^{*,-}(X_i)$ as $K_M^{*,-}(i)$. 

\subsection{Proof of Theorem \ref{thm-new}}\label{sec:proof-thm}


\begin{proof}[Proof of Theorem \ref{thm-new}]
Same as the decomposition in \cite{abadie2011bias}, we write
\begin{align*}
    \hat{\tau}_{M}
    &= \frac{1}{n} \sum_{i=1}^n  \Big[\mu_1(X_i) - \mu_0(X_i)\Big] + \frac{1}{n} \sum_{i=1}^n (2D_i-1) \Big(1 + \frac{K_M(i)}{M}\Big) \big(Y_i - \mu_{D_i}(X_i) \big) \\
    & ~~+ \frac{1}{n}\sum_{i=1}^n (2D_i-1) \Big[\frac{1}{M}\sum_{j \in \mathcal{J}_M(i)} \Big\{ \mu_{1-D_i}(X_i)-\mu_{1-D_i}(X_{j})\Big\} \Big]\\
    & =: \bar{\tau}(\mX) + E_M + B_M.
\end{align*}
It remains to show
\begin{align*}
    R_M & := \sqrt{n}\Big(\bar{\tau}(\mX) + E_M - \frac{1}{n} \sum_{i=1}^n \chi_0(O_i) \Big) \\
    & =\frac{1}{\sqrt{n}} \sum_{i=1}^n \Big[ (2D_i-1)\Big(1+\frac{K_M(i)}{M}\Big) - \Big(\frac{D_i}{e(X_i)} - \frac{1-D_i}{1-e(X_i)} \Big)\Big] \big(Y_i - \mu_{D_i}(X_i)\big)\\
    &= o_{\P_O}(1). 
\end{align*}

We only show that 
\begin{align*}
R_{M,1}:=&\frac{1}{\sqrt{n}} \sum_{D_i=1} \Big[ (2D_i-1)\Big(1+\frac{K_M(i)}{M}\Big) - \Big(\frac{D_i}{e(X_i)} - \frac{1-D_i}{1-e(X_i)} \Big)\Big] \big(Y_i - \mu_{D_i}(X_i)\big)\\
=&\frac{1}{\sqrt{n}} \sum_{D_i=1} \Big[ 1+\frac{K_M(i)}{M} - \frac{1}{e(X_i)} \Big] \big(Y_i - \mu_{1}(X_i)\big)\\
=&o_{\P_O}(1),
\end{align*}
and those units with $D_i=0$ can be treated in a similar way. 

Noticing that $\E[R_{M,1} \given \mX, \mD ]=0$, one has
\begin{align*}
    \E\big[R_{M,1}^2\big] &= \Var\big[R_{M,1}\big] = \E\big[\Var[R_{M,1} \given \mX, \mD] \big] \\
    &= \E\Big[\frac{1}{n} \sum_{i=1, D_i=1}^n \Big(1+\frac{K_M(i)}{M} - \frac{1}{e(X_i)} \Big)^2 \sigma_{1}^2(X_i) \Big] \le C_{\sigma} \E\Big[ 1+\frac{K_M(i)}{M} - \frac{1}{e(X_i)} \Big]^2 = o(1)
\end{align*}
since all cross-terms have zero conditional means and the last equality here is obtained by applying Theorem B.2 in \cite{lin2023estimation}. This completes the proof.
\end{proof}


\subsection{Proof of Theorem \ref{thm2}}

To prove Theorem \ref{thm2}, we first present a lemma that gives an alternative representation of the bootstrap estimator defined in Section \ref{sec:bootstrap}. This lemma corresponds to Equation (3.1) and Lemma 3.1 in \cite{lin2023estimation}.

\begin{lemma}\label{lemma:BTest,rewrite}
    The bias-uncorrected bootstrap matching estimator $\hat{\tau}_{M}^{*}$ can be written as
    \begin{align*}
        \hat{\tau}_{M}^{*} := \frac{1}{n} \sum_{i=1}^n W_i (2D_i-1) \Big(1 + \frac{K^*_M(i)}{M}\Big) Y_i.
    \end{align*}
    Similarly, the bias-corrected bootstrap matching estimator $\hat{\tau}_{M}^{*,{\rm bc}}$ can be written as
    \begin{align*}
        \hat{\tau}_{M}^{*,{\rm bc}} = \frac{1}{n} \sum_{i=1}^n W_i \big[\hat{\mu}_1(X_i) - \hat{\mu}_0(X_i)\big] + \frac{1}{n} \sum_{i=1}^n W_i (2D_i-1) \Big(1 + \frac{K^*_M(i)}{M}\Big) \big( Y_i - \hat{\mu}_{D_i}(X_i) \big).
    \end{align*}
    Here, we write $W_i$ as the bootstrap weight for $O_i$ and $K_M^*(i)$ is some nonnegative integer that lies between $K_M^{*,-}(i)$ and $K_M^{*,+}(i)$ for all $i \in \zahl{n}$.
    \end{lemma}

In addition, conditioning on $\mD$, if the assumptions in Theorem \ref{thm:BT,risk,lp} are satisfied by either $(\P_{X\given D=0}, \P_{X\given D=1})$ or  $(\P_{X\given D=1}, \P_{X\given D=0})$, then for any positive integer $p$, we have
\begin{align*}
    \E_{OW} \Big[ \Big\lvert \frac{K_M^{*,+}(i)}{M} - \frac{1-e(X_i)}{e(X_i)}  \Big\rvert^p \given D_i=1, \mD \Big] \vee \E_{OW} \Big[ \Big\lvert \frac{K_M^{*,+}(i)}{M} - \frac{e(X_i)}{1-e(X_i)}  \Big\rvert^p \given D_i=0, \mD \Big] &= o(1), \\
    \E_{OW} \Big[ \Big\lvert \frac{K_M^{*,-}(i)}{M} - \frac{1-e(X_i)}{e(X_i)}  \Big\rvert^p \given D_i=1, \mD \Big] \vee \E_{OW} \Big[ \Big\lvert \frac{K_M^{*,-}(i)}{M} - \frac{e(X_i)}{1-e(X_i)}  \Big\rvert^p \given D_i=0, \mD \Big] &= o(1).
\end{align*}
They jointly imply 
\begin{align*}
    \E_{OW} \Big[ \Big\lvert \frac{K_M^{*}(i)}{M} - \frac{1-e(X_i)}{e(X_i)}  \Big\rvert^p \given D_i=1, \mD \Big] \vee \E_{OW} \Big[ \Big\lvert \frac{K_M^{*}(i)}{M} - \frac{e(X_i)}{1-e(X_i)}  \Big\rvert^p \given D_i=0, \mD \Big] = o(1). \yestag \label{eq:BootKMconsistency}
\end{align*}

Before proving Theorem \ref{thm2}, we give one more lemma for bounding the moments of 
\[
U_{M,i}^* := \cX_{(M)}^*(X_i^*) - X_i^*. 
\]
This corresponds to Lemma C.2 in \cite{lin2023estimation}.

\begin{lemma}\label{lemma:mbc,BTdist}
  Let $p$ be any positive integer and suppose $M \to \infty$, and $\log (n/M)/M \to 0$ as $n \to \infty$. There then exists a constant $C_p>0$ only depending on $p$ such that, for any $i \in \zahl{n}$, 
  \begin{align*}
    \Big(\frac{n_{1-D_i}}{M}\Big)^{\frac{p}{d}} \E_{OW} \Big[\Big\| U_{M,i}^* \Big\|^p \Biggiven \mD\Big] \le C_p \yestag \label{eq:matching,boundedDist,goal}
  \end{align*}
  holds true with probability $1$.
\end{lemma}

Now we are in place to prove the main theorem.

\begin{proof}[Proof of Theorem \ref{thm2}]

We divide our proof into two parts.

{\bf Part I.} 
In the first part, we prove Theorem \ref{thm2}\ref{thm2:1}. To this end, we decompose $\hat{\tau}_M^*$ as 
\begin{align*}
  \hat{\tau}_{M}^{*} 
  &= \frac{1}{n} \sum_{i=1}^n  W_i \Big[\mu_1(X_i) - \mu_0(X_i)\Big] + \frac{1}{n} \sum_{i=1}^n W_i (2D_i-1) \Big(1 + \frac{K^*_M(i)}{M}\Big) \big(Y_i - \mu_{D_i}(X_i) \big) \\
  & ~~~+ \frac{1}{n}\sum_{i=1}^n (2D_i^*-1) \Big[\frac{1}{M}\sum_{j \in \mathcal{J}_M^*(i)} \Big\{\mu_{1-D_i^*}(X_i^*)-\mu_{1-D_i^*}(X_{j}^*)\Big\}\Big]\\
  & =: \bar{\tau}(\mX)^* + E_M^* + B_M^*.
\end{align*}
It then suffices to show 
\begin{align*}
    R_M^* & := \sqrt{n}\Big(\bar{\tau}(\mX)^* + E_M^* - \frac{1}{n} \sum_{i=1}^n W_i \chi_0(O_i) \Big) \\
    & =\frac{1}{\sqrt{n}} \sum_{i=1}^n W_i \Big[ (2D_i-1)\Big(1+\frac{K_M^*(i)}{M}\Big) - \Big(\frac{D_i}{e(X_i)} - \frac{1-D_i}{1-e(X_i)} \Big)\Big] \big(Y_i - \mu_{D_i}(X_i)\big)\\
    &= o_{\P_{OW}}(1). \yestag \label{eq:matching,BEMALin}
\end{align*}
Similar to the proof of Theorem \ref{thm-new}, we only have to show that the sum of treated units, $R_{M,1}^*$, is $o_{\P_{OW}}(1)$, and the sum of those units with $D_i=0$ can be analyzed similarly. 

Observe that $\E_{OW}[R_{M,1}^* \given \mW, \mX, \mD ]=0$, and the independence relations given by the bootstrap design are
\begin{align*}
    Y_i &\text{ is independent of } Y_j  \given (X_i,D_i), (X_j, D_j), \\ 
    Y_i &\text{ is independent of } \mW \given (X_i,D_i), (X_j, D_j), Y_j, \\
    Y_i &\text{ is independent of } Y_j, \mW \given (X_i,D_i), (X_j, D_j), \\
    Y_i &\text{ is independent of } Y_j \given \mW, (X_i,D_i), (X_j, D_j),
\end{align*}
for any $i\neq j$ with $D_i=D_j=1$. This implies that 

\begin{align*}
    \E_{OW} \big[R_{M,1}^{*2}\big] &= \Var_{OW} \big[R_{M,1}^*\big] \\
    &= \E_{OW} \big[\Var_{OW} [R_{M,1}^* \given \mW, \mX, \mD] \big] \\
    &= \E_{OW} \Big[\frac{1}{n} \sum_{i=1, D_i=1}^n W_{i}^2 \Big(1+\frac{K_M^*(i)}{M} - \frac{1}{e(X_i)} \Big)^2 \sigma_{1}^2(X_i) \Big] \\
    &\le C_{\sigma} \E_{OW} \Big[ W_{i} \Big(1+\frac{K_M^*(i)}{M} - \frac{1}{e(X_i)}\Big) \Big]^2 \\
    &= o(1)
\end{align*}
since all cross-terms of $R_{M,1}^{*2}$ have conditional zero means so that we can apply \eqref{eq:BootKMconsistency} combined with Hölder's inequality, and use the fact that $W_{i}$'s have bounded moments. This completes the first part of our proof.
 
{\bf Part II.}  We now prove Theorem~\ref{thm2}\ref{thm2:2}. This is separated to three more parts. 

{\bf Part II.1.} We first show that
\[
\sqrt{n} (\hat \tau_M^{*, {\rm bc}}-\hat\tau_M^{\rm bc}) = \frac{1}{\sqrt{n}} \sum_{i=1}^n (W_i-1) \chi_0(O_i) + o_{\P_{OW}}(1).
\]
We have, under conditions on the accuracy of $\hat \mu_\omega$ and the smoothness condition on $\mu_\omega$ for $\omega \in \{0,1\}$, 
\begin{align*}
    \hat \tau_M^{*, {\rm bc}} = \bar{\tau}(\mX)^* + E_M^* + B_M^* - \hat B_M^*,
\end{align*}
and we can establish 
\begin{align*}
    \sqrt{n} \big(B_M^* - \hat{B}_M^* \big) = o_{\P_{OW}}(1).
\end{align*}
To see this, write
\begin{align*}
  & \big\lvert B_M^* - \hat{B}_M^* \big\rvert \\
   =& \Big\lvert \frac{1}{n}\sum_{i=1}^n (2D_i^*-1) \Big[\frac{1}{M}\sum_{j \in \mathcal{J}_M^*(i)} \Big\{ \mu_{1-D_i^*}(X_i^*)-\mu_{1-D_i^*}(X_{j}^*) - \hat{\mu}_{1-D_i^*}(X_i^*) + \hat{\mu}_{1-D_i^*}(X_{j}^*)\Big\} \Big] \Big\rvert\\
   \le& \max_{i \in \zahl{n}} \frac{1}{M}\sum_{j \in \mathcal{J}_M^*(i)} \Big\lvert \mu_{1-D_i^*}(X_i^*)-\mu_{1-D_i^*}(X_{j}^*) - \hat{\mu}_{1-D_i^*}(X_i^*) + \hat{\mu}_{1-D_i^*}(X_{j}^*) \Big\rvert\\
   \le& \max_{i \in \zahl{n}, j \in \mathcal{J}_M^*(i), \omega \in \{0,1\}} \Big\lvert \mu_\omega(X_i^*)-\mu_\omega(X_{j}^*) - \hat{\mu}_\omega(X_i^*) + \hat{\mu}_\omega(X_{j}^*) \Big\rvert. \yestag \label{eq:matching,BTconsistency,DiffBM}
\end{align*}
Let $k$ equal to $\floor{d/2}{}+1$ and denote $X_j^* - X_i^*$ by $U_{j,i}^*$ for all $i,j \in \zahl{n}$. Employing Taylor expansion of $\mu_\omega(\cdot)$ around $X_i^*$ to the $k$-th order for $\omega \in \{0,1\}$, one obtains
\begin{align*}
    \Big\lvert \mu_\omega (X_{j}^*) - \mu_{\omega} (X_i^*) - \sum_{\ell = 1}^{k-1} \frac{1}{\ell!} \sum_{t \in \Lambda_\ell} \partial^t \mu_{\omega} (X_i^*) U^{*t}_{j,i} \Big\rvert \le \max_{t \in \Lambda_k} \lVert \partial^t \mu_{\omega} \rVert_\infty \frac{1}{k!} \sum_{t \in \Lambda_k} \lVert U_{j,i}^* \rVert^k, \yestag \label{eq:matching,BTconsistency,TaylorMu}
\end{align*}
and similarly
\begin{align*}
  \Big\lvert \hat{\mu}_\omega (X_{j}^*) - \hat{\mu}_{\omega} (X_i^*) - \sum_{\ell = 1}^{k-1} \frac{1}{\ell!} \sum_{t \in \Lambda_\ell} \partial^t \hat{\mu}_{\omega} (X_i^*) U^{*t}_{j,i} \Big\rvert \le \max_{t \in \Lambda_k} \lVert \partial^t \hat{\mu}_{\omega} \rVert_\infty \frac{1}{k!} \sum_{t \in \Lambda_k} \lVert U_{j,i}^* \rVert^k, \yestag
  \label{eq:matching,BTconsistency,TaylorMuhat}
\end{align*}
with the derivative terms satisfying
\begin{align*}
  \Big\lvert \sum_{\ell = 1}^{k-1} \frac{1}{\ell!} \sum_{t \in \Lambda_\ell} (\partial^t \hat{\mu}_{\omega} (X_i^*) - \partial^t \mu_{\omega} (X_i^*))  U^{*t}_{j,i} \Big\rvert \le \sum_{\ell = 1}^{k-1} \max_{t \in \Lambda_\ell} \lVert \partial^t \hat{\mu}_{\omega} - \partial^t \mu_{\omega} \rVert_\infty \frac{1}{\ell!} \sum_{t \in \Lambda_\ell} \lVert U_{j,i}^* \rVert^\ell. \yestag \label{eq:matching,BTconsistency,DiffTaylor}
\end{align*}
Combining \eqref{eq:matching,BTconsistency,DiffBM}-\eqref{eq:matching,BTconsistency,DiffTaylor}, from Assumption \ref{asp:matching2} one obtains
\begin{align*}
    \big\lvert B_M^* - \hat{B}_M^* \big\rvert & \le \Big( \max_{\omega \in \{0,1\}} \max_{ t \in \Lambda_k} \lVert \partial^t \mu_{\omega} \rVert_\infty + \max_{\omega \in \{0,1\}} \max_{t \in \Lambda_k} \lVert \partial^t \hat{\mu}_{\omega} \rVert_\infty \Big)  \max_{i \in \zahl{n}} \lVert U_{M,i}^* \rVert^k \\
    & ~~ + \sum_{\ell = 1}^{k-1} \Big( \max_{\omega \in \{0,1\}} \max_{t \in \Lambda_\ell} \lVert \partial^t \hat{\mu}_{\omega} - \partial^t \mu_{\omega} \rVert_\infty \Big) \max_{i \in \zahl{n}} \lVert U_{M,i}^* \rVert^\ell \\
    & \lesssim \big(1+O_{\P_{O}}(1) \big) \max_{i \in \zahl{n}} \lVert U_{M,i}^* \rVert^k + \max_{\ell \in \zahl{k-1}} O_{\P_{O}}(n^{-\gamma_\ell}) \max_{i \in \zahl{n}} \lVert U_{M,i}^* \rVert^\ell.
\end{align*}
Using Lemma \ref{lemma:mbc,BTdist} and Markov inequality, for any $\ell \in \zahl{k}$, one has 
\[
\max_{i \in \zahl{n}} \lVert U_{M,i}^* \rVert^\ell = O_{\P_{OW}} ( (M/n)^{\ell/d}).
\]
Then, by the order of $M$, we obtain
\begin{align*}
    \big\lvert B_M^* - \hat{B}_M^* \big\rvert &\lesssim O_{\P_{OW}} \Big(\Big(\frac{M}{n}\Big)^{\frac{k}{d}}\Big) + \max_{\ell \in \zahl{k-1}} O_{\P_{OW}}\Big( n^{-\gamma_\ell} \Big(\frac{M}{n}\Big)^{\frac{\ell}{d}}\Big) \\
    &= o_{\P_{OW}}\big( n^{(\gamma-1)\frac{k}{d}}\big) + \max_{\ell \in \zahl{k-1}} o_{\P_{OW}}\big( n^{(\gamma-1)\frac{\ell}{d}-\gamma_{\ell}}\big) \\
    &= o_{\P_{OW}}\big( n^{-\frac{1}{2}} \big)
\end{align*}
as desired.

{\bf Part II.2.} We then move on to the weak convergence part. 
Writing 
\[
\hat \sigma^2 := \frac{1}{n}\sum_{i=1}^n (\chi_0(O_i) - \bar \chi_0)^2,~~\text{with }\bar\chi_0:=\frac1n\sum_{i=1}^n\chi_0(O_i), 
\]
it is enough to show
\begin{align*}
    \sup_{t \in \bR} \Big\lvert \P_{\mW|\mO}\Big( \frac{1}{\sqrt{n}} \sum_{i=1}^n (W_i-1) \chi_0(O_i) \le t \Big) - \Phi\Big(\frac{t}{\hat \sigma}\Big) \Big\rvert + 
    \sup_{t \in \bR} \Big\lvert \Phi\Big(\frac{t}{\hat \sigma}\Big) - \Phi\Big(\frac{t}{\sigma}\Big) \Big\rvert = o(1) \yestag \label{eq:BTconsistency,a.s.}
\end{align*}
for almost all $\{O_i\}_{i=1}^n$. For the second term in \eqref{eq:BTconsistency,a.s.} to be $o(1)$ almost surely, it follows from $\hat \sigma$ converges to $\sigma$ almost surely (i.e., $\hat\sigma\stackrel{\sf a.s.}{\to}\sigma$) by the strong law of large numbers and continuous mapping theorem, along with the fact that $\Phi(\cdot)$, the cumulative distribution function for the standard normal distribution, is uniformly continuous. 

For the first term in \eqref{eq:BTconsistency,a.s.}, we leverage the Berry-Essen bound \citep[Theorem 3.4.9]{durrett2019probability} conditioning on the original data to obtain
\begin{align*}
    \sup_{t \in \bR} \Big\lvert \P_{\mW|\mO}\Big( \frac{1}{\sqrt{n}} \sum_{i=1}^n (W_i-1) \chi_0(O_i) \le t \Big) - \Phi\Big(\frac{t}{\hat \sigma}\Big) \Big\rvert \le \frac{4}{5\sqrt{n}} \frac{\E_{\mW |\mO}\big[\lvert \chi_0(O_i^*) - \bar \chi_0 \rvert^3 \big] }{ \Var_{\mW |\mO}[\chi_0(O_i^*)]^{3/2} }.
\end{align*}
Furthermore we can bound
\begin{align*}
    \frac{4}{5\sqrt{n}} \frac{\E_{\mW |\mO}\big[\lvert \chi_0(O_i^*) - \bar \chi_0  \rvert^3 \big] }{ \Var_{\mW |\mO}[\chi_0(O_i^*)]^{3/2} } &= \frac{4}{5\sqrt{n}}
    \frac{ n^{-1} \sum_{i=1}^n \lvert \chi_0(O_i) - \bar \chi_0  \rvert^3 }{ \hat \sigma^3 } \\
    & \le \frac{16}{5n^{3/2} \hat \sigma^3 } \Big[ \sum_{i=1}^n \big\lvert \chi_0(O_i) - \E[\chi_0(O_i)] \big\rvert^3 + n \big\lvert \E[\chi_0(O_i)] - \bar \chi_0 \big\rvert^3 \Big].
\end{align*}
Again, by the strong law of large numbers, one has $\bar \chi_0 \stackrel{\sf a.s.}{\longrightarrow} \E[\chi_0(O_i)]$ and $\sigma>0$. The second term above then satisfies 
\[
\big\lvert \E[\chi_0(O_i)] - \bar \chi_0 \big\rvert^3 / \sqrt{n} \hat \sigma^3 \stackrel{\sf a.s.}{\longrightarrow} 0. 
\]
Moreover, Lemma \ref{lemma:MZ-SLLN} implies the first term here also converges to zero almost surely, as we can take $p=2/3$ and note $\E[ \chi_0(O_i)^2] < \infty$. This establishes \eqref{eq:BTconsistency,a.s.}.

{\bf Part II.3.} Lastly, leveraging Lemma \ref{lemma:StoOrders}, we have both $o_{\P_O}(1)$ and $o_{\P_{OW}}(1)$ above are of order $o_{\P_W}(1)$ in $\P_O$-probability. Accordingly, Theorem \ref{thm2}\ref{thm2:2} follows  from Theorem \ref{thm:lin} and Polya theorem \citep[Theorem 2.6.1]{lehmann1999elements}. 
\end{proof}

\subsection{Proof of Corollary \ref{coro1}}

\begin{proof}[Proof of Corollary \ref{coro1}]
    It suffices to show that, when $d=1$,
    \begin{align*}
        n^{1/2} \lvert B_M \rvert = O_{\P_O}\big(M/n^{1/2}\big) ~~\text{ and }~~ n^{1/2} \lvert B_M^* \rvert = O_{\P_{OW}}\big(M/n^{1/2}\big). \yestag \label{eq:BiasOrders}
    \end{align*}
    The first part in \eqref{eq:BiasOrders} corresponds to Corollary 4.1 in \cite{lin2023estimation} and we include a proof here for facilitating the bootstrap part argument. 
    
    For any $d$, using Hölder's inequality implies 
    \begin{align*}
        \Big(\frac{n}{M}\Big)^{2/d} \E[ B_M ]^2 &\le \Big(\frac{n}{M}\Big)^{2/d} \max_{i \in \zahl{n}, j \in \cJ_M(i), \omega \in \{0,1\}}\E \big[ \lvert \mu_{\omega}(X_i)-\mu_{\omega}(X_{j}) \rvert^2 \big]  \\
        & \le \Big(\frac{n}{M}\Big)^{2/d} C_1 \max_{i \in \zahl{n}, \omega \in \{0,1\}}\E \big[ \lVert U_{M,i} \rVert^2 \big], \yestag \label{eq:BM,univ,bd}
    \end{align*}
    where $C_1$ is some constant given by the Lipschitz condition on $\mu_0$ and $\mu_1$, and $U_{M,i}:= \cX_{(M)}(X_i)-X_i$. Since by Lemma C.2 in \cite{lin2023estimation} we know that \eqref{eq:BM,univ,bd} is uniformly bounded over $i$, it then holds true that 
    \[
    B_M = O_{\P_O}((M/n)^{1/d}).
    \] 
    For the second part of \eqref{eq:BiasOrders}, we similarly write 
    \begin{align*}
        \Big(\frac{n}{M}\Big)^{2/d} \E[ B_M^* ]^2 &\le \Big(\frac{n}{M}\Big)^{2/d} \max_{i \in \zahl{n}, j \in \cJ_M^*(i), \omega \in \{0,1\}}\E \big[ \lvert \mu_{\omega}(X_i^*)-\mu_{\omega}(X_{j}^*) \rvert^2 \big]  \\
        & \le \Big(\frac{n}{M}\Big)^{2/d} C_1 \max_{i \in \zahl{n}, \omega \in \{0,1\}}\E \big[ \lVert U_{M,i}^* \rVert^2 \big]. \yestag \label{eq:BM*,univ,bd}
    \end{align*}
    Leveraging Lemma \ref{lemma:mbc,BTdist}, we can also bound \eqref{eq:BM*,univ,bd} and then $B_M^* = O_{\P_{OW}}((M/n)^{1/d})$. This completes our proof in combination with Lemma \ref{lemma:StoOrders}.
\end{proof}

\appendix

\section{The rest proofs}

\subsection{Proofs of results in Section \ref{sec:dr}}

This section gives proofs for all results in Section \ref{sec:dr}.

\paragraph*{Additional notations.} We use $\mX$ and $\mZ$ to represent the matrices composed by  $(X_i)_{i=1}^{N_0}$ and $(Z_i)_{i=1}^{N_1}$, respectively. Let $U[0,1]$ denote the uniform distribution over $[0,1]$. Let $U \sim U[0,1]$ and, for all $k \in \zahl{N_0}$, let $U_{(k)}$ be the $k$-th order statistic of $N_0$ independent random variables sampled from $U[0,1]$, assumed to be mutually independent and all independent of $(\mX,\mZ)$.

It is well known that $U_{(k)}$ follows the beta distribution ${\rm Beta}(k, N_0+1-k)$. Let ${\rm Bin}(\cdot,\cdot)$ denote the binomial distribution and ${\rm Ber}(\cdot)$ denote the Bernoulli distribution. For any positive integer $p$, let $L^p(\bR^d)$ denote the space of all measurable functions $f:\bR^d\to\bR$ such that $\int|f(x)|^p \d x<\infty$. For any $x \in \bR^d$ and measurable function $f:\bR^d \to \bR$, we say $x$ is a Lebesgue point of $f$ if
\[
  \lim_{\delta \to 0^+} \frac{1}{\lambda(B_{x,\delta})} \int_{B_{x,\delta}} \lvert f(x) - f(y) \rvert \d y = 0.
\]

\begin{proof}[Proof of Lemma \ref{lemma:moment,Bootstrap catch}]

Lebesgue differentiation theorem \citep{stein2009real} yields that, for any function $f$ that is integrable on $\bR^d$, it holds true that for $\lambda$-almost all $x$, $x$ is a Lebesgue point of $f$. Therefore, for $\nu_0$-almost all $x$, we have $f_0(x)>0$ and $x$ is a Lebesgue point for both $f_0$ and $f_1$; the last step is due to the absolute continuity of $\nu_0$ and $\nu_1$. 

Similar to \cite{lin2023supp}, it then suffices to consider only those $x \in \bR^d$ such that $f_0(x)>0$ and $x$ is a Lebesgue point of both $f_0$ and $f_1$. {In the following, we only give proofs for the statements concerning $A_M^{*,+}(x)$, as the proof for $A_M^{*,-}(x)$ can be simply obtained by changing some of the "$\leq$"'s to "$<$"'s in the arguments below.} We divide our proof in several parts.

{\bf Part I.} This part proves \eqref{eq:BTDR,local,L1}.

{\bf Case I.1.} Suppose $f_1(x)>0$.

First, we observe that for any input $z$, the bootstrap $M$-th NN $\cX_{(M)}^*(z)$ returns the value of $x \in\{X_i\}_{i=1}^{N_0}$ such that 
\begin{align*}
    \sum_{i=1}^{N_0} W_i \ind\big(\lVert X_i - z \rVert < \lVert x - z \rVert\big) < M ~~\text{ and }~ \sum_{i=1}^{N_0} W_i \ind\big(\lVert X_i - z \rVert \le \lVert x - z \rVert\big) \ge M.
\end{align*}
Conditioning on the original sample $\{X_i\}_{i=1}^{N_0}$ and re-ordering it to be $\{X_i(z)\}_{i=1}^{N_0}$ according to the proximity of each $X_i$ to $z$ in $\bR^d$, we then have for any $k \in \zahl{N_0}$, $\cX_{(M)}^*(z) = \cX_{(k)}(z)$ if and only if the sum of the first $k-1$ nearest $X_i(z)$'s bootstrap weights is strictly smaller than $M$ and the sum of the first $k$ nearest $X_i(z)$'s bootstrap weights exceeds $M$. That is,
\begin{align*}
    S_{k-1}(z) := \sum_{i=1}^{k-1} W_i(z) < M ~~\text{ and }~ S_{k}(z) := \sum_{i=1}^{k} W_i(z) \ge M.
\end{align*}
Denote the probability of this event by $p_k$. From the exchangeability of the Efron's bootstrap weights, we have 
\[
S_{k-1}(z) \sim {\rm Bin}(N_0, (k-1)/N_0)~~ \text{and}~~ S_{k}(z) \sim {\rm Bin}(N_0, k/N_0), 
\]
and therefore
\begin{align*}
    p_k &= \P \big(S_{k-1}(z) < M, S_{k}(z) \ge M \big) \\
    &= \P \big(S_{k}(z) \ge M \big) - \P \big(S_{k-1}(z) \ge M \big) \\ 
    &= \bigg[ 1 - \sum_{j=0}^{M-1} \binom{N_0}{j} \Big(\frac{k}{N_0}\Big)^j \Big(1-\frac{k}{N_0}\Big)^{N_0-j} \bigg] - \bigg[ 1 - \sum_{j=0}^{M-1} \binom{N_0}{j} \Big(\frac{k-1}{N_0}\Big)^j \Big(1-\frac{k-1}{N_0}\Big)^{N_0-j} \bigg] \\
    &= \sum_{j=0}^{M-1} \binom{N_0}{j} \bigg[ \Big(\frac{k-1}{N_0}\Big)^j \Big(1-\frac{k-1}{N_0}\Big)^{N_0-j} - \Big(\frac{k}{N_0}\Big)^j \Big(1-\frac{k}{N_0}\Big)^{N_0-j} \bigg].
\end{align*}

Now we have 
\begin{align*}
    \E_{OW} [\nu_1(A_M^{*,+}(x))] &= \P_{OW} \big( Z \in A_M^{*,+}(x) \big) \\
    &= \P_{OW} \big( \lVert x - Z \rVert \le \lVert \cX_{(M)}^*(Z) - Z \rVert \big) \\
    &= \P_{OW} \big( \nu_0(B_{Z,\lVert x - Z \rVert}) \le \nu_0(B_{Z,\lVert \cX_{(M)}^*(Z) - Z \rVert}) \big). \yestag \label{eq:DR,L1}
\end{align*}
Moreover, for any given $z \in \bR^d$, by the probability integral transformation, one has $\{\nu_0(B_{z,\lVert X_i - z \rVert})\}_{i=1}^{N_0}$ are independent and identically distributed (i.i.d.) from the uniform distribution $U[0,1]$ and 
\[
\nu_0(B_{z,\lVert \cX_{(M)}^*(z) - z \rVert}) = \nu_0(B_{z,\lVert \cX_{(k)}(z) - z \rVert}) \overset{d}{=} U_{(k)}
 \]
 has probability $p_k$ for all $k \in \zahl{N_0}$, which implies $\nu_0(B_{Z,\lVert X - Z \rVert})$, $\nu_0(B_{Z,\lVert \cX_{(M)}^*(Z) - Z \rVert})$ have the same distributions conditioning on $Z$ and that $\nu_0(B_{Z,\lVert X - Z \rVert})$, $\nu_0(B_{Z,\lVert \cX_{(M)}^*(Z) - Z \rVert})$ are independent of $Z$. Also, we have $\nu_1(B_{x,\lVert x-Z \rVert}) \sim U[0,1]$ follows from the same reasoning.

Since $x$ is a Lebesgue point of $\nu_0$ and $\nu_1$, from Lemma \ref{lemma:leb,p}, for any given $\epsilon \in (0,1)$, we can find a $\delta_x>0$ such that all $z \in \bR^d$ satisfing $\lVert x-z \rVert \le \delta_x$ enjoys
\begin{align*}
  & \Big\lvert \frac{\nu_0(B_{x,\lVert z-x \rVert})}{\lambda(B_{x,\lVert z-x \rVert})} - f_0(x) \Big\rvert \le \epsilon f_0(x),~~ \Big\lvert \frac{\nu_0(B_{z,\lVert z-x \rVert})}{\lambda(B_{z,\lVert z-x \rVert})} - f_0(x) \Big\rvert \le \epsilon f_0(x),\\
  & \Big\lvert \frac{\nu_1(B_{x,\lVert z-x \rVert})}{\lambda(B_{x,\lVert z-x \rVert})} - f_1(x) \Big\rvert \le \epsilon f_1(x),~~ \Big\lvert \frac{\nu_1(B_{z,\lVert z-x \rVert})}{\lambda(B_{z,\lVert z-x \rVert})} - f_1(x) \Big\rvert \le \epsilon f_1(x);
  \yestag \label{eq:DR,LebPoint_1}
\end{align*}
and therefore
\begin{align*}
    \frac{1-\epsilon}{1+\epsilon} \frac{f_0(x)}{f_1(x)} \le \frac{\nu_0(B_{z,\lVert x-z \rVert})}{\lambda(B_{z,\lVert x-z \rVert})} \frac{\lambda(B_{x,\lVert x-z \rVert})}{\nu_1(B_{x,\lVert x-z \rVert})} = \frac{\nu_0(B_{z,\lVert x-z \rVert})}{\nu_1(B_{x,\lVert x-z \rVert})} \le \frac{1+\epsilon}{1-\epsilon} \frac{f_0(x)}{f_1(x)}. 
\end{align*}
Moreover, if $\lVert x-z \rVert > \delta_x$, by \eqref{eq:DR,LebPoint_1} applied to $z'$, the intersection point of the line connecting $x$ and $z$ and the boundary of $B_{x,\delta_x}$, we then have 
\begin{align*}
    \nu_0(B_{z,\lVert z-x \rVert}) \ge \nu_0(B_{z',\delta_x}) \ge (1-\epsilon)f_0(x) \lambda(B_{z',\delta_x}) = (1-\epsilon)f_0(x) \lambda(B_{0,\delta_x}). \yestag \label{eq:DR,LebPoint_3}
\end{align*}
Since $M\log N_0/N_0 \to 0$, letting $\eta_N = 4\log(N_0/M)$, we can take $N_0$ sufficiently large such that 
\begin{align*}
    \eta_N \frac{M}{N_0} = 4\frac{M}{N_0} \log\Big(\frac{N_0}{M}\Big) < (1-\epsilon)f_0(x) \lambda(B_{0,\delta_x}); \yestag \label{eq:eta_n}
\end{align*}
in this case, any $z \in \bR^d$ satisfying $\nu_0(B_{z,\lVert z-x \rVert}) \le \eta_N M / N_0$ would satisfy $\lVert x-z \rVert \le \delta_x$.

\vspace{0.2cm}

{\bf Upper bound.} Now we derive the upper bound for \eqref{eq:DR,L1}.
\begin{align*}
    & \E_{OW} [\nu_1(A_M^{*,+}(x))] = \P_{OW} \big( \nu_0(B_{Z,\lVert x - Z \rVert}) \le \nu_0(B_{Z,\lVert \cX_{(M)}^*(Z) - Z \rVert}) \big)  \\
    & ~ \le \P_{OW} \Big( \nu_0(B_{Z,\lVert x - Z \rVert}) \le \nu_0(B_{Z,\lVert \cX_{(M)}^*(Z) - Z \rVert}) \le \eta_N \frac{M}{N_0}\Big) + \P_{OW} \Big(\nu_0(B_{Z,\lVert \cX_{(M)}^*(Z) - Z \rVert}) > \eta_N \frac{M}{N_0}\Big).\yestag \label{eq:BTDR,catch1,uprbd}
\end{align*}
For the second term of \eqref{eq:BTDR,catch1,uprbd}, we can choose 
\[
K = \frac{M}{2} \log\Big(\frac{N_0}{M}\Big), 
\]
and write
\begin{align*}
     \P_{OW} \Big(\nu_0(B_{Z,\lVert \cX_{(M)}^*(Z) - Z \rVert}) > \eta_N \frac{M}{N_0}\Big) &= \sum_{k=1}^{N_0} \P \Big(U_{(k)} > \eta_N \frac{M}{N_0}\Big) p_k \\
    & = \sum_{k \le K} \P \Big(U_{(k)} > \eta_N \frac{M}{N_0}\Big) p_k + \sum_{k > K} \P \Big(U_{(k)} > \eta_N \frac{M}{N_0}\Big) p_k. \yestag \label{eq:BTDR,catch1,uprbd_1}
\end{align*}
Notice that for any $K$ such that $K/M \to \infty$ as $N_0 \to \infty$, by Chernoff bound we have
\begin{align*}
    \sum_{k > K} p_k & = \sum_{j=0}^{M-1} \binom{N_0}{j} \Big(\frac{K}{N_0}\Big)^j \Big(1-\frac{K}{N_0}\Big)^{N_0-j}
    \le \P \Big( {\rm Bin}\Big(N_0, \frac{K}{N_0}\Big) \le M \Big) \\
    & \le \exp\Big\{ M \Big[1-\frac{K}{M} + \log\Big(\frac{K}{M}\Big)\Big] \Big\}
    \le \exp\Big\{ -\frac{1}{2}K \Big\} \yestag \label{eq:BTDR,BTweights,uprbd}
\end{align*}
when $N_0$ is sufficiently large. Moreover, since $K \leq M \log(N_0/M)$, by Chernoff bound we have for any $k \leq K$,
\begin{align*}
     \P \Big(U_{(k)} > \eta_N \frac{M}{N_0}\Big) & \le \P \Big(U_{(K)} > \eta_N \frac{M}{N_0}\Big) = \P\Big({\rm Bin}\Big(N_0, \eta_N \frac{M}{N_0}\Big) \le K \Big) \\
     & \le \exp\Big\{ K \Big[1-\frac{\eta_N M}{K} + \log\Big(\frac{\eta_N M}{K} \Big)\Big] \Big\} \le \exp\Big\{ -\frac{1}{2}K \Big\}. \yestag \label{eq:BTDR,catch1,uprbd_2}
\end{align*}
Combining \eqref{eq:BTDR,catch1,uprbd_1}-\eqref{eq:BTDR,catch1,uprbd_2}, we conclude that
\begin{align*}
    & \lim_{N_0\to\infty} \frac{N_0}{M} \P_{OW} \Big(\nu_0(B_{Z,\lVert \cX_{(M)}^*(Z) - Z \rVert}) > \eta_N \frac{M}{N_0}\Big) \le \lim_{N_0\to\infty} \frac{N_0}{M} \Big[\P \Big(U_{(K)} > \eta_N \frac{M}{N_0}\Big) + \sum_{k > K} p_k \Big] \\
    & ~ \le \lim_{N_0\to\infty} \frac{N_0}{M} \exp\Big\{ -\frac{1}{2}K \Big\} = \lim_{N_0\to\infty} \Big(\frac{N_0}{M}\Big)^{1-M/4} = 0.
\end{align*}
For the first term of \eqref{eq:BTDR,catch1,uprbd}, we have
\begin{align*}
    &\P_{OW} \Big( \nu_0(B_{Z,\lVert x - Z \rVert}) \le \nu_0(B_{Z,\lVert \cX_{(M)}^*(Z) - Z \rVert}) \le \eta_N \frac{M}{N_0}\Big) \\
    & ~ =\P_{OW} \Big( \nu_0(B_{Z,\lVert x - Z \rVert}) \le \nu_0(B_{Z,\lVert \cX_{(M)}^*(Z) - Z \rVert}) \le \eta_N \frac{M}{N_0}, \lVert Z - x \rVert \le \delta_x \Big) \\
    & ~ \le \P_{OW} \big( \nu_0(B_{Z,\lVert x - Z \rVert}) \le \nu_0(B_{Z,\lVert \cX_{(M)}^*(Z) - Z \rVert}), \lVert Z - x \rVert \le \delta_x \big) \\
    & ~ \le \P_{OW} \Big( \frac{1-\epsilon}{1+\epsilon} \frac{f_0(x)}{f_1(x)}  \nu_1(B_{x,\lVert x-Z \rVert}) \le \nu_0(B_{Z,\lVert \cX_{(M)}^*(Z) - Z \rVert}), \lVert Z - x \rVert \le \delta_x \Big) \\
    & ~ \le \P_{OW} \Big( \frac{1-\epsilon}{1+\epsilon} \frac{f_0(x)}{f_1(x)}  \nu_1(B_{x,\lVert x-Z \rVert}) \le \nu_0(B_{Z,\lVert \cX_{(M)}^*(Z) - Z \rVert}) \Big) = \P_{OW} \Big( \frac{1-\epsilon}{1+\epsilon} \frac{f_0(x)}{f_1(x)} U \le \nu_0(B_{Z,\lVert \cX_{(M)}^*(Z) - Z \rVert}) \Big) \\
    & ~ = \sum_{k=1}^{N_0} \P \Big( \frac{1-\epsilon}{1+\epsilon} \frac{f_0(x)}{f_1(x)} U \le U_{(k)} \Big) p_k =  \sum_{k=1}^{N_0} \Big[ \int_0^1 \P \Big( U_{(k)} \ge \frac{1-\epsilon}{1+\epsilon} \frac{f_0(x)}{f_1(x)}  t \Big)\d t \Big] p_k.
\end{align*}
Therefore, we can upper bound the first term by
\begin{align*}
    & \frac{N_0}{M} \P_{OW} \Big( \nu_0(B_{Z,\lVert x - Z \rVert}) \le \nu_0(B_{Z,\lVert \cX_{(M)}^*(Z) - Z \rVert}) \le \eta_N \frac{M}{N_0}\Big) \\
    & ~ \le \frac{N_0}{M} \sum_{k=1}^{N_0} \Big[ \int_0^1 \P \Big( U_{(k)} \ge \frac{1-\epsilon}{1+\epsilon} \frac{f_0(x)}{f_1(x)}  t \Big)\d t \Big] p_k = \frac{N_0}{M} \frac{1+\epsilon}{1-\epsilon} \frac{f_1(x)}{f_0(x)} \sum_{k=1}^{N_0} \Big[ \int_0^{\frac{1-\epsilon}{1+\epsilon} \frac{f_0(x)}{f_1(x)}} \P \big( U_{(k)} \ge  t \big)\d t \Big] p_k \\
    & ~ \le \frac{N_0}{M} \frac{1+\epsilon}{1-\epsilon} \frac{f_1(x)}{f_0(x)} \sum_{k=1}^{N_0} \Big[ \int_0^{\infty} \P \big( U_{(k)} \ge  t \big)\d t \Big] p_k = \frac{N_0}{M} \frac{1+\epsilon}{1-\epsilon} \frac{f_1(x)}{f_0(x)} \sum_{k=1}^{N_0} \E\big[ U_{(k)} \big] p_k.
\end{align*}
Now, we observe that the summation term in the product can be written as
\begin{align*}
    \sum_{k=1}^{N_0} \E\big[ U_{(k)} \big] p_k &= \sum_{k=1}^{N_0} \frac{k}{N_0+1} \Big\{ \sum_{j=0}^{M-1} \binom{N_0}{j} \Big[ \Big(\frac{k-1}{N_0}\Big)^j \Big(1-\frac{k-1}{N_0}\Big)^{N_0-j} - \Big(\frac{k}{N_0}\Big)^j \Big(1-\frac{k}{N_0}\Big)^{N_0-j} \Big] \Big\} \\
    & = \sum_{k=1}^{N_0} \frac{1}{N_0+1} \Big\{ \sum_{j=0}^{M-1} \binom{N_0}{j}  \Big(\frac{k-1}{N_0}\Big)^j \Big(1-\frac{k-1}{N_0}\Big)^{N_0-j} \Big\} \\
    & = \frac{1}{N_0+1} \sum_{k=1}^{N_0} \P\Big( {\rm Bin}\Big(N_0, \frac{k-1}{N_0}\Big) < M \Big) = \frac{1}{N_0+1} \sum_{k=1}^{N_0} \P\Big( U_{(M)} \ge \frac{k-1}{N_0} \Big). \yestag \label{eq:matching,DRlocal,boundedSum_1_1}
\end{align*}
Further observing
\begin{align*}
    & \bigg \lvert \frac{1}{N_0} \sum_{k=1}^{N_0} \P\Big( U_{(M)} \ge \frac{k-1}{N_0} \Big) - \int_0^1 \P\big(U_{(M)} \ge t \big) \d t \bigg \rvert \\
    = & \bigg \lvert \sum_{k=1}^{N_0} \int_{(k-1)/N_0}^{k/N_0} \Big\{ \P\Big( U_{(M)} \ge \frac{k-1}{N_0} \Big) -  \P\big(U_{(M)} \ge t \big) \Big\} \d t \bigg \rvert \\
    = & \bigg \lvert \sum_{k=1}^{N_0} \int_{0}^{1/N_0} \Big\{ \P\Big( U_{(M)} \ge \frac{k-1}{N_0} \Big) -  \P\Big( U_{(M)} \ge \frac{k-1}{N_0} + t \Big) \Big\} \d t \bigg \rvert \\
    = &  \bigg \lvert \int_{0}^{1/N_0} \sum_{k=1}^{N_0} \Big\{ \P\Big( U_{(M)} \ge \frac{k-1}{N_0} \Big) -  \P\Big( U_{(M)} \ge \frac{k-1}{N_0} + t \Big) \Big\} \d t \bigg \rvert \\
    \le & \sup_{0 \le t \le 1/N_0} \sum_{k=1}^{N_0} \bigg\lvert \P\Big( U_{(M)} \ge \frac{k-1}{N_0} \Big) -  \P\Big( U_{(M)} \ge \frac{k-1}{N_0} + t \Big) \bigg\rvert \int_{0}^{1/N_0} \d t \\
    = & \sup_{0 \le t \le 1/N_0}  \sum_{k=1}^{N_0} \P\Big( \frac{k-1}{N_0} \le U_{(M)} < \frac{k-1}{N_0} + t \Big) \int_{0}^{1/N_0} \d t \le \frac{1}{N_0},
\end{align*}
we then know that
\begin{align*}
    & \limsup_{N_0\to\infty} \frac{N_0}{M} \P_{OW} \Big( \nu_0(B_{Z,\lVert x - Z \rVert}) \le \nu_0(B_{Z,\lVert \cX_{(M)}^*(Z) - Z \rVert}) \le \eta_N \frac{M}{N_0}\Big) \\
    \le& \limsup_{N_0\to\infty} \frac{N_0}{M} \frac{1+\epsilon}{1-\epsilon} \frac{f_1(x)}{f_0(x)} \Big( \E\big[ U_{(M)} \big] + O\Big(\frac{1}{N_0}\Big)  \Big) = \frac{1+\epsilon}{1-\epsilon} \frac{f_1(x)}{f_0(x)}. \yestag \label{eq:matching,DRlocal,boundedSum_1_2}
\end{align*}
Combining the two terms together, we obtain, for any $\epsilon>0$,
\begin{align*}
     \limsup_{N_0\to\infty} \frac{N_0}{M} \E_{OW} [\nu_1(A_M^{*,+}(x))] \le \frac{1+\epsilon}{1-\epsilon} \frac{f_1(x)}{f_0(x)}. \yestag \label{eq:BTDR,catch1,uprbd1-1}
\end{align*}

\vspace{0.2cm}

{\bf Lower bound.} Now we derive the lower bound for \eqref{eq:DR,L1}. First, we note that, by the choice of $\eta_N$, $\lVert Z - x \rVert \le \delta_x$ would be implied by $\frac{1+\epsilon}{1-\epsilon} \frac{f_0(x)}{f_1(x)}  \nu_1(B_{x,\lVert x-Z \rVert}) \le \eta_N \frac{M}{N_0}$ as when $\lVert Z - x \rVert > \delta_x$, 
\begin{align*}
    \frac{1+\epsilon}{1-\epsilon} \frac{f_0(x)}{f_1(x)}  \nu_1(B_{x,\lVert x-Z \rVert}) &> \frac{1+\epsilon}{1-\epsilon} \frac{f_0(x)}{f_1(x)}  \nu_1(B_{x,\delta_x})  \\ 
    & \ge \frac{1+\epsilon}{1-\epsilon} \frac{f_0(x)}{f_1(x)} (1-\epsilon)f_1(x) \lambda(B_{x,\delta_x}) = (1+\epsilon)f_0(x) \lambda(B_{0,\delta_x}) > \eta_N \frac{M}{N_0}.
\end{align*}
This allows us to write
\begin{align*}
  & \E_{OW} [\nu_1(A_M^{*,+}(x))] = \P_{OW} \big( \nu_0(B_{Z,\lVert x - Z \rVert}) \le \nu_0(B_{Z,\lVert \cX_{(M)}^*(Z) - Z \rVert}) \big) \\
  & ~\ge \P_{OW} \Big( \nu_0(B_{Z,\lVert x - Z \rVert}) \le \nu_0(B_{Z,\lVert \cX_{(M)}^*(Z) - Z \rVert})  \le \eta_N \frac{M}{N_0} \Big) \\
  & ~= \P_{OW} \Big( \nu_0(B_{Z,\lVert x - Z \rVert}) \le \nu_0(B_{Z,\lVert \cX_{(M)}^*(Z) - Z \rVert})  \le \eta_N \frac{M}{N_0}, \lVert Z - x \rVert \le \delta_x \Big)\\
  & ~\ge \P_{OW} \Big( \frac{1+\epsilon}{1-\epsilon} \frac{f_0(x)}{f_1(x)}  \nu_1(B_{x,\lVert x-Z \rVert}) \le \nu_0(B_{Z,\lVert \cX_{(M)}^*(Z) - Z \rVert})  \le \eta_N \frac{M}{N_0}, \lVert Z - x \rVert \le \delta_x \Big)\\
  & ~= \P_{OW} \Big( \frac{1+\epsilon}{1-\epsilon} \frac{f_0(x)}{f_1(x)}  \nu_1(B_{x,\lVert x-Z \rVert}) \le \nu_0(B_{Z,\lVert \cX_{(M)}^*(Z) - Z \rVert}) \le \eta_N \frac{M}{N_0} \Big)\\
  & ~\ge \P_{OW} \Big( \frac{1+\epsilon}{1-\epsilon} \frac{f_0(x)}{f_1(x)}  \nu_1(B_{x,\lVert x-Z \rVert}) \le \nu_0(B_{Z,\lVert \cX_{(M)}^*(Z) - Z \rVert}) \Big) - \P_{OW} \Big( \nu_0(B_{Z,\lVert \cX_{(M)}^*(Z) - Z \rVert}) > \eta_N \frac{M}{N_0} \Big)\\
  & ~= \P_{OW} \Big( \frac{1+\epsilon}{1-\epsilon} \frac{f_0(x)}{f_1(x)} U \le \nu_0(B_{Z,\lVert \cX_{(M)}^*(Z) - Z \rVert})\Big) - \P_{OW} \Big( \nu_0(B_{Z,\lVert \cX_{(M)}^*(Z) - Z \rVert}) > \eta_N \frac{M}{N_0}\Big). \yestag \label{eq:BTDR,catch1,lwrbd}
\end{align*}
For the second term here, from the previous analysis we have shown
\begin{align*}
    & \lim_{N_0\to\infty} \frac{N_0}{M} \P_{OW} \Big(\nu_0(B_{Z,\lVert \cX_{(M)}^*(Z) - Z \rVert}) > \eta_N \frac{M}{N_0}\Big) = 0.
\end{align*}
For the first term in \eqref{eq:BTDR,catch1,lwrbd}, we write it as
\begin{align*}
    & \P_{OW} \Big( \frac{1+\epsilon}{1-\epsilon} \frac{f_0(x)}{f_1(x)} U \le \nu_0(B_{Z,\lVert \cX_{(M)}^*(Z) - Z \rVert})\Big) = \sum_{k=1}^{N_0} \P \Big( \frac{1+\epsilon}{1-\epsilon} \frac{f_0(x)}{f_1(x)} U \le U_{(k)}\Big) p_k \\
    & ~= \sum_{k=1}^{N_0} \Big[\int_0^1 \P \Big( U_{(k)} \ge \frac{1+\epsilon}{1-\epsilon} \frac{f_0(x)}{f_1(x)}  t \Big)\d t \Big] p_k = \frac{1-\epsilon}{1+\epsilon} \frac{f_1(x)}{f_0(x)} \sum_{k=1}^{N_0} \Big[\int_0^{\frac{1+\epsilon}{1-\epsilon} \frac{f_0(x)}{f_1(x)}} \P \big( U_{(k)} \ge t \big)\d t \Big] p_k \\
    & ~= \frac{1-\epsilon}{1+\epsilon} \frac{f_1(x)}{f_0(x)} \sum_{k=1}^{N_0} \Big[\int_0^{1} \P \big( U_{(k)} \ge t \big)\d t + \int_1^{\frac{1+\epsilon}{1-\epsilon} \frac{f_0(x)}{f_1(x)}} \P \big( U_{(k)} \ge t \big)\d t \Big]p_k.
\end{align*}
By an analogous argument to our prior analysis of \eqref{eq:matching,DRlocal,boundedSum_1_1} and \eqref{eq:matching,DRlocal,boundedSum_1_2}, one can show 
\begin{align*}
    & \lim_{N_0\to\infty} \frac{N_0}{M} \frac{1-\epsilon}{1+\epsilon} \frac{f_1(x)}{f_0(x)} \sum_{k=1}^{N_0} \Big[\int_0^{1} \P ( U_{(k)} \ge t)\d t \Big]p_k \\
     =& \lim_{N_0\to\infty} \frac{N_0}{M} \frac{1-\epsilon}{1+\epsilon} \frac{f_1(x)}{f_0(x)} \sum_{k=1}^{N_0} \E \big[U_{(k)} \big] p_k \\
     =& \frac{1-\epsilon}{1+\epsilon} \frac{f_1(x)}{f_0(x)}.
\end{align*}
For the second term, supposing $\frac{1+\epsilon}{1-\epsilon} \frac{f_0(x)}{f_1(x)}<1$ and choosing $K = N_0/\log N_0$, by Chernoff's inequality we have for $N_0$ sufficiently large
\begin{align*}
    & \sum_{k=1}^{N_0} \Big[ \int_{\frac{1+\epsilon}{1-\epsilon} \frac{f_0(x)}{f_1(x)}}^{1} \P ( U_{(k)} \ge t )\d t \Big]p_k \le \Big[1-\frac{1+\epsilon}{1-\epsilon} \frac{f_0(x)}{f_1(x)}\Big] \sum_{k=1}^{N_0} \P \Big( U_{(k)} \ge \frac{1+\epsilon}{1-\epsilon} \frac{f_0(x)}{f_1(x)} \Big) p_k \\
    & ~=
    \Big[1-\frac{1+\epsilon}{1-\epsilon} \frac{f_0(x)}{f_1(x)}\Big] \Big[ \sum_{k \le K} \P \Big( U_{(k)} \ge \frac{1+\epsilon}{1-\epsilon} \frac{f_0(x)}{f_1(x)} \Big) p_k + \sum_{k > K} \P \Big( U_{(k)} \ge \frac{1+\epsilon}{1-\epsilon} \frac{f_0(x)}{f_1(x)} \Big) p_k \Big] \\
    & ~\le \Big[1-\frac{1+\epsilon}{1-\epsilon} \frac{f_0(x)}{f_1(x)}\Big] \Big[ \P \Big( U_{(K)} \ge \frac{1+\epsilon}{1-\epsilon} \frac{f_0(x)}{f_1(x)} \Big) + \sum_{k > K} p_k \Big] \\
    & ~\le \Big[1-\frac{1+\epsilon}{1-\epsilon} \frac{f_0(x)}{f_1(x)}\Big] \Big[ \exp\Big\{ K \Big[ 1- \frac{1+\epsilon}{1-\epsilon} \frac{f_0(x)}{f_1(x)} \frac{N_0}{K} + \log \Big(\frac{1+\epsilon}{1-\epsilon} \frac{f_0(x)}{f_1(x)} \frac{N_0}{K}\Big) \Big] \Big\} + \exp\Big\{ - \frac{1}{2} K\Big\} \Big] \\
    & ~\le \Big[1-\frac{1+\epsilon}{1-\epsilon} \frac{f_0(x)}{f_1(x)}\Big] 2 \exp\Big\{ - \frac{1}{2} K\Big\} \le \Big[1-\frac{1+\epsilon}{1-\epsilon} \frac{f_0(x)}{f_1(x)}\Big] 2 \exp\Big\{ - \frac{1}{2} N_0^{1/2}\Big\},
\end{align*}
which implies
\begin{align*}
    & \lim_{N_0\to\infty} \frac{N_0}{M} \frac{1-\epsilon}{1+\epsilon} \frac{f_1(x)}{f_0(x)} \sum_{k=1}^{N_0} \Big[\int_{\frac{1+\epsilon}{1-\epsilon} \frac{f_0(x)}{f_1(x)}}^1 \P \big( U_{(k)} \ge t \big)\d t \Big]p_k \le \lim_{N_0\to\infty} \frac{N_0}{M} 2\exp\Big\{ - \frac{1}{2} N_0^{1/2}\Big\} = 0.
\end{align*}
Since it is also the case for $\frac{1+\epsilon}{1-\epsilon} \frac{f_0(x)}{f_1(x)} \ge 1$ as $\P \big( U_{(k)} \ge t \big) = 0$ for $t \ge 1$, we can conclude
\begin{align*}
     \liminf_{N_0 \to \infty} \frac{N_0}{M} \E_{OW} [\nu_1(A_M^{*,+}(x))] \ge \frac{1-\epsilon}{1+\epsilon} \frac{f_1(x)}{f_0(x)}
     \yestag \label{eq:BTDR,catch1,lwrbd1-1}
\end{align*}
for any $\epsilon>0$. Combining \eqref{eq:BTDR,catch1,uprbd1-1} and \eqref{eq:BTDR,catch1,lwrbd1-1} we then obtain
\begin{align*}
    \lim_{N_0 \to \infty} \frac{N_0}{M} \E_{OW} [\nu_1(A_M^{*,+}(x))] =  \frac{f_1(x)}{f_0(x)} =r(x).
\end{align*}

\vspace{0.2cm}

{\bf Case I.2.} Suppose now $f_1(x)=0$. 

For any given $\epsilon \in (0,1)$, from Lemma \ref{lemma:leb,p} we can find a $\delta_x>0$ such that all $z \in \bR^d$ that satisfies $\lVert x-z \rVert \le \delta_x$ one has
\begin{align*}
  \Big\lvert \frac{\nu_0(B_{z,\lVert z-x \rVert})}{\lambda(B_{z,\lVert z-x \rVert})} - f_0(x) \Big\rvert  \le \epsilon f_0(x)~~~{\rm and}~~~ \Big\lvert \frac{\nu_1(B_{x,\lVert z-x \rVert})}{\lambda(B_{x,\lVert z-x \rVert})} \Big\rvert \le \epsilon,
\end{align*}
and therefore
\begin{align*}
    \nu_0(B_{z,\lVert x-z \rVert}) \ge (1-\epsilon) f_0(x) \lambda(B_{z,\lVert x-z \rVert}) = (1-\epsilon) f_0(x) \lambda(B_{x,\lVert x-z \rVert}) \ge \frac{1-\epsilon}{\epsilon} f_0(x) \nu_1(B_{x,\lVert x-z \rVert}).
\end{align*}
Notice \eqref{eq:DR,LebPoint_3} only uses the first half of \eqref{eq:DR,LebPoint_1}; taking $\eta_N$ the same as in \eqref{eq:eta_n}, then for all $z \in \bR^d$ satisfying $\nu_0(B_{z,\lVert z-x \rVert}) \le \eta_N M / N_0$ would again implies $\lVert x-z \rVert \le \delta_x$.  Similar argument as in Case I.1 for the upper bound then yields
\begin{align*}
    & \E_{OW} [\nu_1(A_M^{*,+}(x))] = \P_{OW} \big( \nu_0(B_{Z,\lVert x - Z \rVert}) \le \nu_0(B_{Z,\lVert \cX_{(M)}^*(Z) - Z \rVert}) \big)  \\
    & ~ \le \sum_{k=1}^{N_0} \P \Big( \frac{1-\epsilon}{\epsilon} f_0(x) U \le U_{(k)} \Big) p_k + \P_{OW} \Big(\nu_0(B_{Z,\lVert \cX_{(M)}^*(Z) - Z \rVert}) > \eta_N \frac{M}{N_0}\Big);
\end{align*}
and the first term above can be bounded by
\begin{align*}
    & \frac{N_0}{M} \sum_{k=1}^{N_0} \P \Big( \frac{1-\epsilon}{\epsilon} f_0(x) U \le U_{(k)} \Big) p_k = \frac{N_0}{M} \sum_{k=1}^{N_0} \Big[\int_0^1 \P \Big( U_{(k)} \ge \frac{1-\epsilon}{\epsilon} f_0(x) t \Big)\d t \Big] p_k \\
    &=  \frac{N_0}{M} \frac{\epsilon}{1-\epsilon} \frac{1}{f_0(x)} \sum_{k=1}^{N_0} \Big[\int_0^{\frac{1-\epsilon}{\epsilon} f_0(x)} \P \big( U_{(k)} \ge t \big)\d t \Big] p_k \le \frac{N_0}{M} \frac{\epsilon}{1-\epsilon} \frac{1}{f_0(x)} \sum_{k=1}^{N_0} \E \big[U_{(k)} \big] p_k \\
    & = \frac{\epsilon}{1-\epsilon} \frac{1}{f_0(x)} \frac{N_0}{N_0+1} + o(1)
\end{align*}
from \eqref{eq:matching,DRlocal,boundedSum_1_1} and \eqref{eq:matching,DRlocal,boundedSum_1_2}. The arbitrariness of $\epsilon$ then implies
\begin{align*}
    \lim_{N_0 \to \infty} \frac{N_0}{M} \E_{OW} [\nu_1(A_M^{*,+}(x))] = 0 =r(x),
\end{align*}
which completes the proof of \eqref{eq:BTDR,local,L1} for both cases.

\vspace{0.2cm}

{\bf Part II.} This part proves \eqref{eq:BTDR,local,Lp}.

{\bf Case II.1.} Suppose $f_1(x)>0$. 
For any $\epsilon \in (0,1)$ with $\delta_x$ samely defined as in Case I.1, we take $\eta_N = 4p \log(N_0/M)$ and $N_0$ sufficiently large such that
\begin{align*}
    \eta_N \frac{M}{N_0} = 4p \frac{M}{N_0}\log\Big(\frac{N_0}{M}\Big) < (1-\epsilon) f_0(x) \lambda(B_{0,\delta_x}).
\end{align*}
Let $\tZ_1,\ldots,\tZ_p$ be $p$ independent random vectors that are drawn from $\nu_1$ and further independent of the system. Then
\begin{align*}
  & \E_{OW}\big[\nu_1^p(A_M^{*,+}(x))\big] = \P_{OW} \big(\tZ_1,\ldots,\tZ_p \in A_M^{*,+}(x)\big)\\
  & ~= \P_{OW} \big( \lVert x - \tZ_1 \rVert \le \lVert \cX_{(M)}^* (\tZ_1) - \tZ_1 \rVert, \ldots, \lVert x - \tZ_p \rVert \le \lVert \cX_{(M)}^*(\tZ_p) - \tZ_p \rVert \big)\\
  & ~= \P_{OW} \Big( \nu_0(B_{\tZ_1,\lVert x - \tZ_1 \rVert}) \le \nu_0(B_{\tZ_1,\lVert \cX_{(M)}^*(\tZ_1) - \tZ_1 \rVert}), \ldots, \nu_0(B_{\tZ_p,\lVert x - \tZ_p \rVert}) \le \nu_0(B_{\tZ_p,\lVert \cX_{(M)}^*(\tZ_p) - \tZ_p \rVert}) \Big) \\
  & ~\le \P_{OW} \Big( \max_{j \in \zahl{p}} \nu_0(B_{\tZ_j,\lVert x - \tZ_j \rVert}) \le \max_{j \in \zahl{p}} \nu_0(B_{\tZ_j,\lVert \cX_{(M)}^*(\tZ_j) - \tZ_j \rVert}) \Big) \\
  & ~\le \P_{OW} \Big( \max_{j \in \zahl{p}} \nu_0(B_{\tZ_j,\lVert x - \tZ_j \rVert}) \le \max_{j \in \zahl{p}} \nu_0(B_{\tZ_j,\lVert \cX_{(M)}^*(\tZ_j) - \tZ_j \rVert}) \le \eta_N \frac{M}{N_0} \Big) \\
  & ~~~~+ \P_{OW} \Big( \max_{j \in \zahl{p}} \nu_0(B_{\tZ_j,\lVert \cX_{(M)}^*(\tZ_j) - \tZ_j \rVert}) > \eta_N \frac{M}{N_0} \Big). \yestag \label{eq:DR,Lp,decomp1}
\end{align*}
For the second term above, since it is upper bounded by 
\begin{align*}
    \P_{OW} \Big( \max_{j \in \zahl{p}} \nu_0(B_{\tZ_j,\lVert \cX_{(M)}^*(\tZ_j) - \tZ_j \rVert}) > \eta_N \frac{M}{N_0} \Big) \le p \P_{OW} \Big( \nu_0(B_{\tZ_1,\lVert \cX_{(M)}^*(\tZ_1) - \tZ_1 \rVert}) > \eta_N \frac{M}{N_0} \Big),
\end{align*}
and by previous analysis, we obtain
\begin{align*}
    & \lim_{N_0\to\infty} \Big(\frac{N_0}{M}\Big)^p \P_{OW} \Big( \max_{j \in \zahl{p}} \nu_0(B_{\tZ_j,\lVert \cX_{(M)}^*(\tZ_j) - \tZ_j \rVert}) > \eta_N \frac{M}{N_0} \Big) \\
     \le& \lim_{N_0\to\infty} p \Big(\frac{N_0}{M}\Big)^p \P_{OW} \Big(\nu_0(B_{Z_1,\lVert \cX_{(M)}^*(Z_1) - Z_1 \rVert}) > \eta_N \frac{M}{N_0}\Big) \\
    \le& \lim_{N_0\to\infty} p \Big(\frac{N_0}{M}\Big)^p \exp\Big\{ - \frac{M}{4} \log\Big(\frac{M}{N_0}\Big) \Big\} = \lim_{N_0\to\infty} p \Big(\frac{N_0}{M}\Big)^{p-M/4} \\
  =& 0.
\end{align*}
For the first term in \eqref{eq:DR,Lp,decomp1}, we have
\begin{align*}
    & \P_{OW} \Big( \max_{j \in \zahl{p}} \nu_0(B_{\tZ_j,\lVert x - \tZ_j \rVert}) \le \max_{j \in \zahl{p}} \nu_0(B_{\tZ_j,\lVert \cX_{(M)}^*(\tZ_j) - \tZ_j \rVert}) \le \eta_N \frac{M}{N_0} \Big) \\
    = & \P_{OW} \Big( \max_{j \in \zahl{p}} \nu_0(B_{\tZ_j,\lVert x - \tZ_j \rVert}) \le \max_{j \in \zahl{p}} \nu_0(B_{\tZ_j,\lVert \cX_{(M)}^*(\tZ_j) - \tZ_j \rVert}) \le \eta_N \frac{M}{N_0}, \max_{j \in \zahl{p}} \lVert \tZ_j - x \rVert \le \delta_x \Big) \\
    \le& \P_{OW} \Big( \frac{1-\epsilon}{1+\epsilon} \frac{f_0(x)}{f_1(x)} \max_{j \in \zahl{p}} \nu_1(B_{x,\lVert x-\tZ_j \rVert}) \le \max_{j \in \zahl{p}} \nu_0(B_{\tZ_j,\lVert \cX_{(M)}^*(\tZ_j) - \tZ_j \rVert}) \le \eta_N \frac{M}{N_0}, \max_{j \in \zahl{p}} \lVert \tZ_j - x \rVert \le \delta_x \Big) \\
    \le& \P_{OW} \Big( \frac{1-\epsilon}{1+\epsilon} \frac{f_0(x)}{f_1(x)} \max_{j \in \zahl{p}} \nu_1(B_{x,\lVert x-\tZ_j \rVert}) \le \max_{j \in \zahl{p}} \nu_0(B_{\tZ_j,\lVert \cX_{(M)}^*(\tZ_j) - \tZ_j \rVert}) \Big).
\end{align*}
Notice that $\{ \nu_1(B_{x,\lVert x-\tZ_j \rVert}) \}_{j=1}^p$ are i.i.d. from $U[0,1]$ by probability integral transformation and $\max_{j \in \zahl{p}} \nu_1(B_{x,\lVert x-\tZ_j \rVert}) \overset{d}{=} {\rm Beta}(p,1)$, we then have
\begin{align*}
    & \Big(\frac{N_0}{M}\Big)^p \P_{OW} \Big( \max_{j \in \zahl{p}} \nu_0(B_{\tZ_j,\lVert x - \tZ_j \rVert}) \le \max_{j \in \zahl{p}} \nu_0(B_{\tZ_j,\lVert \cX_{(M)}^*(\tZ_j) - \tZ_j \rVert}) \le \eta_N \frac{M}{N_0} \Big) \\
    \le& \Big(\frac{N_0}{M}\Big)^p \int_0^1 \P_{OW} \Big( \frac{1-\epsilon}{1+\epsilon} \frac{f_0(x)}{f_1(x)} t \le \max_{j \in \zahl{p}} \nu_0(B_{\tZ_j,\lVert \cX_{(M)}^*(\tZ_j) - \tZ_j \rVert}) \Big) pt^{p-1} \d t \\
    =& \Big(\frac{N_0}{M}\Big)^p \Big(\frac{1+\epsilon}{1-\epsilon} \frac{f_1(x)}{f_0(x)}\Big)^p \int_0^{\frac{1-\epsilon}{1+\epsilon}\frac{f_0(x)}{f_1(x)}} \P_{OW} \Big( t \le \max_{j \in \zahl{p}} \nu_0(B_{\tZ_j,\lVert \cX_{(M)}^*(\tZ_j) - \tZ_j \rVert}) \Big) pt^{p-1} \d t.
\end{align*}
To prove 
\begin{align*}
    \limsup_{N_0 \to \infty} \Big(\frac{N_0}{M}\Big)^p \int_0^{\frac{1-\epsilon}{1+\epsilon}\frac{f_0(x)}{f_1(x)}} \P_{OW} \Big( t \le \max_{j \in \zahl{p}} \nu_0(B_{\tZ_j,\lVert \cX_{(M)}^*(\tZ_j) - \tZ_j \rVert}) \Big) pt^{p-1} \d t \le 1, \yestag \label{eq:DR,Lp,le1}
\end{align*}
we first observe that for any fixed $\xi>0$ such that $K=(1+\xi)M$, one has
\begin{align*}
    \sum_{k > K} p_k &\le \P\Big({\rm Bin}\Big(N_0, \frac{K}{N_0}\Big) < M\Big) = \P\Big( U_{(M)} > \frac{K}{N_0}\Big) \le \exp\bigg\{\frac{- \xi M}{\frac{3}{\xi}\big[(\frac{N_0}{N_0+1})(1-\frac{M}{N_0+1}) \big] + 3}\bigg\}
\end{align*}
from applying Lemma 3.1.1 in \cite{reiss2012approximate}. Consequently, for any positive integer $p$, we have
\begin{align*}
     \limsup_{N_0 \to \infty} \Big(\frac{N_0}{M}\Big)^p \sum_{k>(1+\xi)M}  p_k = 0 \yestag \label{eq:matching,DRlocal,upperPbound}
\end{align*}
by the assumption of $\log (N_0/M)/M \to 0$. Similarly, for any $\xi>0$ such that $K=(1-\xi)M$,
\begin{align*}
    \sum_{k < K} p_k &\le \P\Big({\rm Bin}\Big(N_0, \frac{K}{N_0}\Big) \ge M\Big) = \P\Big( U_{(M)} \le \frac{K}{N_0}\Big) \le \exp\bigg\{\frac{- \xi M}{\frac{3}{\xi}\big[(\frac{N_0}{N_0+1})(1-\frac{M}{N_0+1}) \big] + 3}\bigg\};
\end{align*}
and thus for any positive integer $p$,
\begin{align*}
    \limsup_{N_0 \to \infty} \Big(\frac{N_0}{M}\Big)^p \sum_{k<(1-\xi)M}  p_k = 0. \yestag \label{eq:matching,DRlocal,lowerPbound}
\end{align*}
Combining \eqref{eq:matching,DRlocal,upperPbound} and \eqref{eq:matching,DRlocal,lowerPbound}, we then have for any fixed $\xi>0$,
\begin{align*}
     o(1) + (1-\xi)^p \le \Big(\frac{N_0}{M}\Big)^p \sum_{k=1}^{N_0}\Big(\frac{k}{N_0}\Big)^p p_k  \le (1+\xi)^p + o(1);
\end{align*}
therefore the arbitrariness of $\xi$ yields
\begin{align*}
    \lim_{N_0 \to \infty} \Big(\frac{N_0}{M}\Big)^p \sum_{k=1}^{N_0} \Big(\frac{k}{N_0}\Big)^p p_k = 1. \yestag \label{eq:matching,DRlocal,boundedSum}
\end{align*}
Returning to the previous analysis of \eqref{eq:DR,Lp,le1}, we have for any fixed $\xi>0$,
\begin{align*}
    & \Big(\frac{N_0}{M}\Big)^p \int_0^{\frac{1-\epsilon}{1+\epsilon}\frac{f_0(x)}{f_1(x)}} \P_{OW} \Big( t \le \max_{j \in \zahl{p}} \nu_0(B_{\tZ_j,\lVert \cX_{(M)}^*(\tZ_j) - \tZ_j \rVert}) \Big) pt^{p-1} \d t \\
     =& \Big(\frac{N_0}{M}\Big)^p \int_0^{(1+\xi)\frac{M}{N_0}} \P_{OW} \Big( t \le \max_{j \in \zahl{p}} \nu_0(B_{\tZ_j,\lVert \cX_{(M)}^*(\tZ_j) - \tZ_j \rVert}) \Big) pt^{p-1} \d t \\
    & + \Big(\frac{N_0}{M}\Big)^p \int_{(1+\xi)\frac{M}{N_0}}^{\frac{1-\epsilon}{1+\epsilon}\frac{f_0(x)}{f_1(x)}} \P_{OW} \Big( t \le \max_{j \in \zahl{p}} \nu_0(B_{\tZ_j,\lVert \cX_{(M)}^*(\tZ_j) - \tZ_j \rVert}) \Big) pt^{p-1} \d t \\
     \le& \Big(\frac{N_0}{M}\Big)^p \int_0^{(1+\xi)\frac{M}{N_0}}  pt^{p-1} \d t
    + p \Big(\frac{N_0}{M}\Big)^p \int_{(1+\xi)\frac{M}{N_0}}^{\frac{1-\epsilon}{1+\epsilon}\frac{f_0(x)}{f_1(x)}} \P_{OW} \Big( t \le \nu_0(B_{\tZ_1,\lVert \cX_{(M)}^*(\tZ_1) - \tZ_1 \rVert}) \Big) pt^{p-1} \d t \\
    =& (1+\xi)^p + p \Big(\frac{N_0}{M}\Big)^p \sum_{k=1}^{N_0} \Big[ \int_{(1+\xi)\frac{M}{N_0}}^{\frac{1-\epsilon}{1+\epsilon}\frac{f_0(x)}{f_1(x)}} \P ( t \le U_{(k)}) pt^{p-1} \d t \Big] p_k, \yestag \label{eq:DR,Lp,decom2}
\end{align*}
and we can write
\begin{align*}
    & p \Big(\frac{N_0}{M}\Big)^p \sum_{k=1}^{N_0} \Big[ \int_{(1+\xi) \frac{M}{N_0}}^{\frac{1-\epsilon}{1+\epsilon}\frac{f_0(x)}{f_1(x)}} \P ( t \le U_{(k)} ) pt^{p-1} \d t \Big] p_k \\
     \le & p \Big(\frac{N_0}{M}\Big)^p \sum_{k\le (1+\xi)M} \Big[ \int_{(1+\xi) \frac{M}{N_0}}^{\frac{1-\epsilon}{1+\epsilon}\frac{f_0(x)}{f_1(x)}} \P ( t \le U_{(k)} ) pt^{p-1} \d t \Big] p_k + p \Big(\frac{N_0}{M}\Big)^p \sum_{k > (1+\xi)M} \Big[ \int_{(1+\xi) \frac{M}{N_0}}^{\frac{1-\epsilon}{1+\epsilon}\frac{f_0(x)}{f_1(x)}} pt^{p-1} \d t \Big] p_k \\
     \le & p \Big(\frac{N_0}{M}\Big)^p \sum_{k\le (1+\xi)M} \Big[ \int_{\frac{k}{N_0}}^{\frac{1-\epsilon}{1+\epsilon}\frac{f_0(x)}{f_1(x)}} \P ( t \le U_{(k)} ) pt^{p-1} \d t \Big] p_k + p \Big( \frac{1-\epsilon}{1+\epsilon}\frac{f_0(x)}{f_1(x)} \Big)^p \Big(\frac{N_0}{M}\Big)^p \sum_{k>(1+\xi)M} p_k \\
     \le &  p \Big(\frac{N_0}{M}\Big)^p \sum_{k=1}^{N_0} \Big[ \int_{\frac{k}{N_0}}^{\frac{1-\epsilon}{1+\epsilon}\frac{f_0(x)}{f_1(x)}} \P ( t \le U_{(k)} ) pt^{p-1} \d t \Big] p_k + o(1)
\end{align*}
from applying \eqref{eq:matching,DRlocal,upperPbound}. For the first term above, by Chernoff bound applied to the lower tail probability, it has upper bound
\begin{align*}
    & p \Big(\frac{N_0}{M}\Big)^p \sum_{k=1}^{N_0} \Big[ \int_{\frac{k}{N_0}}^{\frac{1-\epsilon}{1+\epsilon}\frac{f_0(x)}{f_1(x)}} \P ( {\rm Bin}(N_0, t ) \le k ) pt^{p-1} \d t \Big] p_k \\
    =& p \Big(\frac{N_0}{M}\Big)^p \sum_{k=1}^{N_0} \Big[ \Big(\frac{k}{N_0}\Big)^{p} \int_{0}^{\frac{1-\epsilon}{1+\epsilon}\frac{f_0(x)}{f_1(x)} \frac{N_0}{k} - 1} \P \Big( {\rm Bin}\Big(N_0, \frac{k}{N_0}(1+t) \Big) \le k \Big) p(1+t)^{p-1} \d t \Big] p_k \\
    \le& p^2 \Big(\frac{N_0}{M}\Big)^p \sum_{k=1}^{N_0} \Big[ \Big(\frac{k}{N_0}\Big)^{p} \int_{0}^{\infty} \exp \big\{k[ 1 - (1+t) + \log(1+t)]\big\} (1+t)^{p-1} \d t \Big] p_k \\
    =& p^2 \Big(\frac{N_0}{M}\Big)^p \sum_{k=1}^{N_0} \Big[ \Big(\frac{k}{N_0}\Big)^{p} \int_{0}^{\infty} \exp \{-tk \} (1+t)^{k+p-1} \d t \Big] p_k \\
    =& p^2 \Big(\frac{N_0}{M}\Big)^p \sum_{k=1}^{N_0} \Big[ \Big(\frac{k}{N_0}\Big)^{p} \exp \{k\} \int_{1}^{\infty} \exp \{-tk \} t^{k+p-1} \d t \Big] p_k \\
    \le& p^2 \Big(\frac{N_0}{M}\Big)^p \sum_{k=1}^{N_0} \Big[ \Big(\frac{k}{N_0}\Big)^{p} \frac{\exp \{k \}}{k^{k+p}} \int_{0}^{\infty} \exp \{-t \} t^{k+p-1} \d t \Big] p_k \\
    =& p^2 \Big(\frac{N_0}{M}\Big)^p \Big\{ \sum_{k \le K} \Big[ \Big(\frac{k}{N_0}\Big)^{p} \frac{\exp \{k \}}{k^{k+p}} \Gamma(k+p) \Big] p_k + \sum_{k > K} \Big[ \Big(\frac{k}{N_0}\Big)^{p} \frac{\exp \{k \}}{k^{k+p}} \Gamma(k+p) \Big] p_k \Big\}.
\end{align*}
For any $k \to \infty$, by Stirling's approximation we get the sequence
\begin{align*}
    \frac{\exp \{k \}}{k^{k+p}} \Gamma(k+p) &= \frac{\exp \{k \}}{k^{k+p}} \Gamma(k+1) (k+1)\cdots(k+p-1) \\
    &\le \sqrt{2\pi k} \frac{(k+1)\cdots(k+p-1)}{k^p} \Big[ 1+\Big(\exp\Big\{\frac{1}{12k}\Big\}-1\Big) \Big] \\
    &\lesssim \sqrt{\frac{2\pi}{k}} \Big[ 1+\Big(\exp\Big\{\frac{1}{12k}\Big\}-1\Big) \Big] = o(1).
\end{align*}
Denoting the upper bound of this sequence by $B$ and choosing $K = M^{1/2}$, we can further upper bound the above term by
\begin{align*}
    & p^2 \Big(\frac{N_0}{M}\Big)^p \bigg\{ \sum_{k \le K} \Big[ \Big(\frac{k}{N_0}\Big)^{p} \frac{\exp \{k \}}{k^{k+p}} \Gamma(k+p) \Big] p_k + \sum_{k > K} \Big[ \Big(\frac{k}{N_0}\Big)^{p} \frac{\exp \{k \}}{k^{k+p}} \Gamma(k+p) \Big] p_k \bigg\} \\
     \le&
    \frac{p^2 B}{M^{p/2}} \sum_{k \le K} p_k +  \frac{\sqrt{2\pi} p^2}{M^{1/4}} \exp\Big\{\frac{1}{12M^{1/2}}\Big\} \Big(\frac{N_0}{M}\Big)^p \sum_{k = 1}^{N_0} \Big[ \Big(\frac{k}{N_0}\Big)^{p} p_k \Big] = o(1).
\end{align*}
Here, the last equality is due to  \eqref{eq:matching,DRlocal,boundedSum}. We then conclude that the second term in \eqref{eq:DR,Lp,decom2} goes to 0 and 
\begin{align*}
    \limsup_{N_0 \to \infty} \Big(\frac{N_0}{M}\Big)^p \E_{OW}[\nu_1^p(A_M^{*,+}(x))] \le (1+\xi)^p \Big(\frac{1+\epsilon}{1-\epsilon} \frac{f_1(x)}{f_0(x)}\Big)^p
\end{align*}
for any $\epsilon \in (0,1)$ and $\xi>0$. On the other hand, for the lower bound we can directly apply Lyapunov inequality to obtain
\begin{align*}
    \liminf_{N_0 \to \infty} \Big(\frac{N_0}{M}\Big)^p \E_{OW}[\nu_1^p(A_M^{*,+}(x))] \ge \liminf_{N_0 \to \infty} \Big(\frac{N_0}{M}\Big) \big[ \E_{OW}[\nu_1 (A_M^{*,+}(x))] \big]^p = \big[ r(x) \big]^p
\end{align*}
from \eqref{eq:BTDR,local,L1}. Therefore, by the arbitrariness of $\epsilon$ and $\xi$, we have
\begin{align*}
    \lim_{N_0\to\infty} \Big(\frac{N_0}{M}\Big)^p \E_{OW}\big[\nu_1^p\big(A_M^{*,+}(x)\big)\big] = \big[r(x)\big]^p.
\end{align*}

\vspace{0.2cm}

{\bf Case II.2.} Suppose $f_1(x)=0$. For any given $\epsilon \in (0,1)$, we take $\delta_x$ same as Case I.2 and $\eta_N$ same as Case II.1. Proceeding similarly as before, one obtains
\begin{align*}
  & \Big(\frac{N_0}{M}\Big)^p \E_{OW}\big[\nu_1^p(A_M^{*,+}(x))\big] \le \Big(\frac{N_0}{M}\Big)^p \P_{OW} \Big( \max_{j \in \zahl{p}} \nu_0(B_{\tZ_j,\lVert x - \tZ_j \rVert}) \le \max_{j \in \zahl{p}} \nu_0(B_{\tZ_j,\lVert \cX_{(M)}^*(\tZ_j) - \tZ_j \rVert}) \Big) \\
  & ~\le \Big(\frac{N_0}{M}\Big)^p \P_{OW} \Big( \max_{j \in \zahl{p}} \nu_0(B_{\tZ_j,\lVert x - \tZ_j \rVert}) \le \max_{j \in \zahl{p}} \nu_0(B_{\tZ_j,\lVert \cX_{(M)}^*(\tZ_j) - \tZ_j \rVert}) \le \eta_N \frac{M}{N_0} \Big) \\
  & ~~~+ \Big(\frac{N_0}{M}\Big)^p \P_{OW} \Big( \max_{j \in \zahl{p}} \nu_0(B_{\tZ_j,\lVert \cX_{(M)}^*(\tZ_j) - \tZ_j \rVert}) > \eta_N \frac{M}{N_0} \Big) \\
  & ~\le \Big(\frac{N_0}{M}\Big)^p \Big(\frac{\epsilon}{1-\epsilon} \frac{1}{f_0(x)}\Big)^p \int_0^{\frac{\epsilon}{1-\epsilon} \frac{1}{f_0(x)}} \P_{OW} \Big( t \le \max_{j \in \zahl{p}} \nu_0(B_{\tZ_j,\lVert \cX_{(M)}^*(\tZ_j) - \tZ_j \rVert}) \Big) pt^{p-1} \d t + o(1);
\end{align*}
and for any fixed $\xi>0$,
\begin{align*}
    \limsup_{N_0 \to \infty}\Big(\frac{N_0}{M}\Big)^p \E_{OW}\big[\nu_1^p(A_M^{*,+}(x))\big] \le (1+\xi)^p \Big(\frac{\epsilon}{1-\epsilon} \frac{1}{f_0(x)}\Big)^p.
\end{align*}
Lastly, by the arbitrariness of $\epsilon$ and $\xi$, we then have 
\begin{align*}
    \lim_{N_0\to\infty} \Big(\frac{N_0}{M}\Big)^p \E_{OW}\big[\nu_1^p\big(A_M^{*,+}(x)\big)\big] = \big[r(x)\big]^p=0.
\end{align*}
Together with Case II.1, this completes the proof for \eqref{eq:BTDR,local,Lp}.
\end{proof}

\begin{proof}[Proof of Theorem \ref{thm:cons,BTlp}]

Same as before, we only give proofs for the statements concerning $\hat{r}_M^{*,+}(x)$, and the proofs for $\hat{r}_M^{*,-}(x)$ can be obtained similarly. 

{\bf Part I.} In this part we prove \eqref{eq:DRest,L1}.
\begin{align*}
    \lim_{N_0,N_1\to\infty} \E_{OW} \big[\hat{r}_M^{*,+}(x)\big] &= \lim_{N_0,N_1\to\infty} \E_{OW} \Big[ \frac{N_0}{N_1} \frac{K_M^{*,+}(x)}{M} \Big] \\
    & = \lim_{N_0,N_1\to\infty} \E_{OW} \Big[ \frac{N_0}{N_1M} \sum_{j=1}^{N_1} W_{Zj} \ind\big(Z_j \in A_M^{*,+}(x)\big) \Big] \\
    &= \lim_{N_0,N_1\to\infty} \frac{N_0}{N_1 M} \sum_{j=1}^{N_1} \E_{OW} [ W_{Zj}] \E_{OW} \big[ \ind\big(Z_j \in A_M^{*,+}(x)\big) \big] \\
    &= \lim_{N_0,N_1\to\infty} \frac{N_0}{N_1M} \sum_{j=1}^{N_1} \E_{OW} \big[ \nu_1 \big( A_M^{*,+}(x)\big) \big] = r(x).
\end{align*}

{\bf Part II.} In this part we prove \eqref{eq:DRest,Lp}. By Lyapunov inequality, it suffices to consider the case where $p$ is even. Since $x^p$ is a convex function, plugging in the identity
\[
\E_{OW} [ \hat{r}_M^{*,+}(x) \given \mX, \mW_{\mX} ] = \frac{N_0}{M}\nu_1 ( A_M^{*,+}(x)),
\]
we have
\begin{align*}
    & \E_{OW} \Big[ \Big| \hat{r}_M^{*,+}(x) - r(x) \Big|^p \Big] \\
    \le& 2^{p-1} \bigg\{ \E_{OW} \Big[ \Big| \hat{r}_M^{*,+}(x) - \E_{OW}[ \hat{r}_M^{*,+}(x) \given \mX, \mW_{\mX} ] \Big|^p \Big] + \E_{OW} \Big[ \Big| \frac{N_0}{M}\nu_1 \big( A_M^{*,+}(x)\big) - r(x) \Big|^p \Big] \bigg\}.
\end{align*}
For the second term above, we can expand the product term and by Lemma \ref{lemma:moment, bootstrapBinomial},
\begin{align*}
    \lim_{N_0\to\infty} \E_{OW} \Big[ \Big| \frac{N_0}{M}\nu_1 \big( A_M^{*,+}(x)\big) - r(x) \Big|^p \Big] &= \lim_{N_0\to\infty} \sum_{j=0}^p \binom{p}{j} \Big\{ \Big(- \frac{N_0}{M}\Big)^j \E_{OW} \big[ \nu_1^j \big( A_M^{*,+}(x)\big) \big]  r(x)^{p-j} \Big\} \\
    &= r(x)^p \sum_{j=0}^p \binom{p}{j} (-1)^j = r(x)^p (1-1)^p = 0.
\end{align*}
For the first term, notice that, conditioning on $( \mX, \mW_{\mX} )$, $K_M^{*,+}(x)$ is the sum of $\mW_{\mZ}$-weighted Bernoulli random variables with parameter $\nu_1 ( A_M^{*,+}(x))$. Therefore, by Lemma \ref{lemma:moment, bootstrapBinomial},
\begin{align*}
    & \E_{OW} \Big[ \Big| K_M^{*,+}(x) - N_1 \nu_1\big( A_M^{*,+}(x)\big) \Big|^p \Big| \mX, \mW_{\mX} \Big] = P_{1,p}^* N_1 \nu_1 \big( A_M^{*,+}(x)\big)\big(1-\nu_1 \big( A_M^{*,+}(x)\big)\big) \\
    & ~~~ + P_{2,p}^* \big[N_1 \nu_1 \big( A_M^{*,+}(x)\big)\big(1-\nu_1 \big( A_M^{*,+}(x)\big)\big)\big]^2 + \ldots + P_{p/2,p}^* \big[N_1 \nu_1 \big( A_M^{*,+}(x)\big)\big(1-\nu_1 \big( A_M^{*,+}(x)\big)\big) \big]^{p/2},
\end{align*}
where $\{P_{j,p}^*\}_{j=1}^{p/2}$ denotes certain polynomials in $\nu_1 ( A_M^{*,+}(x))$ that do not contain $N_1$. For any positive integers $p,q$, and $k$ such that $k \le p$, by Lemma \ref{lemma:moment,Bootstrap catch} we have 
\begin{align*}
    \lim_{N_0 \to \infty} \Big( \frac{N_0}{N_1M}\Big)^p \E_{OW} [N_1^p \nu_1^p(A_M^{*,+}(x))] &= \big[r(x)\big]^p,\\
    \lim_{N_0 \to \infty} \Big( \frac{N_0}{N_1M}\Big)^p \Big( \frac{N_0}{M}\Big)^q \E_{OW} [N_1^p \nu_1^{p+q}(A_M^{*,+}(x))] &= \big[r(x)\big]^{p+q}, \\
    \lim_{N_0 \to \infty} \Big( \frac{N_0}{N_1M}\Big)^k \Big( \frac{N_0}{M}\Big)^q \E_{OW} [N_1^k \nu_1^{k+q}(A_M^{*,+}(x))] &= \big[r(x)\big]^{k+q},
\end{align*}
as $M N_1/N_0 \to \infty$ and $N_0/M \to \infty$. We then conclude that $\E_{OW} [N_1^p \nu_1^p(A_M^{*,+}(x))]$ is the dominating term among $\{ \E_{OW} [N_1^k \nu_1^{k+q}(A_M^{*,+}(x))] \}_{k \le p, q \ge 0}$. We then have
\begin{align*}
    & \E_{OW} \Big[ \Big| \hat{r}_M^{*,+}(x) - \E_{OW}[ \hat{r}_M^{*,+}(x) \given \mX, \mW_{\mX} ] \Big|^p \Big] = \Big( \frac{N_0}{N_1M}\Big)^p \E_{OW} \Big[ \Big| K_M^{*,+}(x) - N_1 \nu_1 \big( A_M^{*,+}(x)\big) \Big|^p  \Big] \\
    & ~ \lesssim \Big( \frac{N_0}{N_1M}\Big)^p \E_{OW} \Big[N_1^{p/2} \nu_1^{p/2}(A_M^{*,+}(x))\Big] \lesssim \Big( \frac{N_0}{N_1M}\Big)^{p/2} = o(1).
\end{align*}
Combining these two terms together, we then have 
\begin{align*}
    \lim_{N_0\to\infty} \E_{OW} \Big[ \lvert \hat{r}_M^{*,+}(x) - r(x) \rvert^p \Big] = 0,
\end{align*}
and thus finish the proof.
\end{proof}

\begin{proof}[Proof of Theorem \ref{thm:BT,risk,lp}]

Same as before, we only give proofs for the statements concerning $\hat{r}_M^{*,+}(x)$. 

By \ref{asp:BT,Lp,risk-1}\ref{asp:BT,Lp,risk-2} of Assumption \ref{asp:BT,Lp,risk}, we know that both $S_0$ and $S_1$ are bounded in $\bR^d$, and hence $\nu_0$ and $\nu_1$ are compactly supported. Also, as $f_0, f_1 \in L^1$ and the class of continuous functions is dense in the class of compactly supported $L_1$ functions; for any given $\epsilon > 0$, by Lusin's theorem \citep[Theorem 4.5]{stein2009real},  we may select uniformly continuous functions $g_0,g_1$ such that $\norm{f_0-g_0}_{L_1} \le \epsilon^3$ and $\norm{f_1-g_1}_{L_1} \le \epsilon^3$. Let $\delta>0$ be the number such that for any $x,z \in \bR^d$ satisfying $\norm{z-x} \le \delta$,
\begin{align*}
    \abs{g_0(x) - g_0(z)}{} \le \epsilon^2/3 ~~~\text{and}~~ \abs{g_1(x) - g_1(z)}{} \le \epsilon^2/3.
\end{align*}
We then have, for any $x \in \bR^d$ and $x \in B$ with ${\rm diam}(B) \le \delta$,
\begin{align*}
    & \frac{1}{\lambda(B)} \int_{B} \abs{f_0(x) - f_0(z)}{} \d z \\
     \le& \frac{1}{\lambda(B)} \int_{B} \abs{f_0(x) - g_0(x)}{} \d z + 
    \frac{1}{\lambda(B)} \int_{B} \abs{g_0(x) - g_0(z)}{} \d z +
    \frac{1}{\lambda(B)} \int_{B} \abs{g_0(z) - f_0(x)}{} \d z.
\end{align*}
For the first term $\abs{f_0(x) - g_0(x)}{}$, Markov's inequality yields
\begin{align*}
    \lambda \Big( \Big\{x: \abs{f_0(x) - g_0(x)}{} > \frac{\epsilon^2}{3}\Big\}\Big) \le \frac{3}{\epsilon^2} \norm{f_0-g_0}_{L_1} \le 3 \epsilon.
\end{align*}
For the second term, uniform continuity gives us 
\begin{align*}
    \frac{1}{\lambda(B)} \int_{B} \abs{g_0(x) - g_0(z)}{} \d z \le \max_{x,z \in B} \abs{g_0(x) - g_0(z)}{} \le \epsilon^2/3.
\end{align*}
For the third term, by the definition of maximal functions in Lemma \ref{lemma:HL}, one has
\begin{align*}
    \frac{1}{\lambda(B)} \int_{B} \abs{g_0(z) - f_0(x)}{} \d z \le {\sf M}(f_0-g_0)(x)
\end{align*}
and 
\begin{align*}
    \lambda \Big( \Big\{x: {\sf M}(f_0-g_0)(x) > \frac{\epsilon^2}{3}\Big\}\Big) \le \frac{3C_d}{\epsilon^2} \norm{f_0-g_0}_{L_1} \le 3C_d \epsilon.
\end{align*}
As similar analysis also applies to $f_1$, we can then define 
\begin{align*}
  A_1 = A_1(\epsilon) := & \Big\{x: \Big\lvert f_0(x) - g_0(x) \Big\rvert > \frac{\epsilon^2}{3} \Big\} \bigcup \Big\{x: \Big\lvert f_1(x) - g_1(x) \Big\rvert > \frac{\epsilon^2}{3} \Big\} \\
  & \bigcup \Big\{x:{\sf M}(f_0-g_0)(x) > \frac{\epsilon^2}{3} \Big\} \bigcup \Big\{x:{\sf M}(f_1-g_1)(x) > \frac{\epsilon^2}{3} \Big\};
\end{align*}
so $\lambda(A_1) \le 6(C_d+1)\epsilon$ and for all $x \in A_1^c$ with ${\rm diam}(B) \le \delta$, we have
\begin{align*}
     \frac{1}{\lambda(B)} \int_{B} \abs{f_0(x) - f_0(z)}{} \d z \le \epsilon^2 ~~~\text{and}~~ \frac{1}{\lambda(B)} \int_{B} \abs{f_1(x) - f_1(z)}{} \d z \le \epsilon^2.
\end{align*}
Further, define 
\[
A_2 = A_2(\epsilon) := \Big\{x: f_1(x)<\epsilon \Big\}
\]
and write the quantity of our interest as
\begin{align*}
     &\E_{OW} \Big[ \int_{\bR^d} \Big\lvert \hat{r}_M^{*,+}(x) - r(x) \Big\rvert^p f_0(x) \d x \Big] \\
     =& \E_{OW} \Big[ \int_{\bR^d} \Big\lvert \hat{r}_M^{*,+} - r(x) \Big\rvert^p \Big\{\ind(x \in A_1^c \cap A_2^c) + \ind(x \in A_1^c \cap A_2) + \ind(x \in A_1 ) \Big\} f_0(x) \d x \Big].
\end{align*}
Since this equality holds for arbitrary $\epsilon>0$, we can restricting our attention to those $\epsilon \le f_L$ and separate the proof in the following three cases.

\vspace{0.2cm}

{\bf Case I.} $x \in A_1^c(\epsilon) \cap A_2^c(\epsilon)$.

For $x \in A_1^c \cap A_2^c$, we get by $\epsilon \le f_L$ and the definition of $A_2$ that
\begin{align*}
     &\frac{1}{\lambda(B)} \int_{B} \abs{f_0(x) - f_0(z)}{} \d z \le \epsilon^2 \le \epsilon f_L \le \epsilon f_0(x) \\
     {\rm and}~~~&\frac{1}{\lambda(B)} \int_{B} \abs{f_1(x) - f_1(z)}{} \d z \le \epsilon^2 \le \epsilon f_1(x),
\end{align*}
which imply, for $\norm{z-x} \le \delta/2$, one has
\begin{align*}
  & \Big\lvert \frac{\nu_0(B_{x,\lVert z-x \rVert})}{\lambda(B_{x,\lVert z-x \rVert})} - f_0(x) \Big\rvert \le \epsilon f_0(x), ~~ \Big\lvert \frac{\nu_0(B_{z,\lVert z-x \rVert})}{\lambda(B_{z,\lVert z-x \rVert})} - f_0(x) \Big\rvert \le \epsilon f_0(x),\\
  & \Big\lvert \frac{\nu_1(B_{x,\lVert z-x \rVert})}{\lambda(B_{x,\lVert z-x \rVert})} - f_1(x) \Big\rvert \le \epsilon f_1(x), ~~\Big\lvert \frac{\nu_1(B_{z,\lVert z-x \rVert})}{\lambda(B_{z,\lVert z-x \rVert})} - f_1(x) \Big\rvert \le \epsilon f_1(x).
\end{align*}
Similarly as Case II.1 in the proof of Lemma \ref{lemma:moment,Bootstrap catch}, for any positive integer $p$, we can let $\eta_N = \eta_{N,p} = 4p \log(N_0/M)$ and take $N_0$ sufficiently large such that for all $x \in S_0$,
\begin{align*}
    \eta_N \frac{M}{N_0} = 4p\frac{M}{N_0} \log\Big(\frac{N_0}{M}\Big) < (1-\epsilon)f_L \lambda(B_{0,\delta}) \le (1-\epsilon)f_0(x)  \lambda(B_{0,\delta}).
\end{align*}
Then, following an argument analogous to Lemma \ref{lemma:moment,Bootstrap catch}, we have the $L^p$-moment for the bootstrap catchment area's $\nu_1$-measure satisfies
\begin{align*}
    \lim_{\epsilon \to 0} \lim_{N_0\to\infty} \Big(\frac{N_0}{M}\Big)^p \E_{OW}\big[\nu_1^p\big(A_M^{*,+}(x)\big)\big] = \big[r(x)\big]^p.
\end{align*}
By selection of $\eta_N$, we know this limit holds uniformly for $x \in A_1^c(\epsilon) \cap A_2^c(\epsilon)$. Proceeding as the proof of Theorem \ref{thm:cons,BTlp}, this implies
\begin{align*}
    \lim_{\epsilon \to 0} \lim_{N_0\to\infty} \sup_{x \in A_1^c \cap A_2^c} \E_{OW} \Big[ \lvert \hat{r}_M^{*,+}(x) - r(x) \rvert^p \Big] = 0.
\end{align*}
Finally, by Fubini's theorem and Fatou's Lemma we have
\begin{align*}
    \lim_{\epsilon \to 0} \lim_{N_0\to\infty} \E_{OW} \Big[ \int_{\bR^d} \Big\lvert \hat{r}_M^{*,+} - r(x) \Big\rvert^p \ind(x \in A_1^c(\epsilon) \cap A_2^c(\epsilon)) f_0(x) \d x \Big] = 0.
\end{align*}

\vspace{0.2cm}

{\bf Case II.} $x \in A_1^c(\epsilon) \cap A_2(\epsilon)$.

For this case, we have
\begin{align*}
  & \Big\lvert \frac{\nu_0(B_{x,\lVert z-x \rVert})}{\lambda(B_{x,\lVert z-x \rVert})} - f_0(x) \Big\rvert \le \epsilon f_0(x),~~ \Big\lvert \frac{\nu_0(B_{z,\lVert z-x \rVert})}{\lambda(B_{z,\lVert z-x \rVert})} - f_0(x) \Big\rvert \le \epsilon f_0(x),\\
  & \Big\lvert \frac{\nu_1(B_{x,\lVert z-x \rVert})}{\lambda(B_{x,\lVert z-x \rVert})} - f_1(x) \Big\rvert \le \epsilon^2,~~ \Big\lvert \frac{\nu_1(B_{z,\lVert z-x \rVert})}{\lambda(B_{z,\lVert z-x \rVert})} - f_1(x) \Big\rvert \le \epsilon^2.
\end{align*}
Take $N_0$ sufficiently large such that $\eta_N$ is the same as Case I. Proceeding similarly to the proof of Case II.2 in Lemma~\ref{lemma:moment,Bootstrap catch}, we have
\begin{align*}
    \lim_{\epsilon \to 0} \lim_{N_0\to\infty} \sup_{x \in A_1^c \cap A_2} \E_{OW} \Big[ \lvert \hat{r}_M^{*,+}(x) - r(x) \rvert^p \Big] = 0.
\end{align*}
Fubini's theorem and Fatou's Lemma then yield
\begin{align*}
    \lim_{\epsilon \to 0} \lim_{N_0\to\infty} \E_{OW} \Big[ \int_{\bR^d} \Big\lvert \hat{r}_M^{*,+} - r(x) \Big\rvert^p \ind(x \in A_1^c(\epsilon) \cap A_2(\epsilon)) f_0(x) \d x \Big] = 0.
\end{align*}

\vspace{0.2cm}

{\bf Case III.} $x \in A_1(\epsilon)$.

In this case, by Assumption \ref{asp:BT,Lp,risk}\ref{asp:BT,Lp,risk-3}\ref{asp:BT,Lp,risk-4}, we know that for any $x \in S_0$ and $z \in S_1$, 
\begin{align*}
    \nu_0(B_{z,\lVert z-x \rVert}) \ge f_L \lambda(B_{z,\lVert z-x \rVert} \cap S_0) \ge af_L \lambda(B_{z,\lVert z-x \rVert}) \ge \frac{af_L}{f_U} \nu_1(B_{x,\lVert z-x \rVert}).
\end{align*}
We can deal with the $L^p$-moment for the bootstrap catchment area's $\nu_1$-measure similarly as case II in the proof of Lemma \ref{lemma:moment,Bootstrap catch}:
\begin{align*}
  & \Big(\frac{N_0}{M}\Big)^p \E_{OW}\big[\nu_1^p\big(A_M^{*,+}(x)\big)\big] = \Big(\frac{N_0}{M}\Big)^p \P_{OW} \big(\tZ_1,\ldots,\tZ_p \in A_M^{*,+}(x)\big)\\
  & ~= \Big(\frac{N_0}{M}\Big)^p \P_{OW} \big( \lVert x - \tZ_1 \rVert \le \lVert \cX_{(M)}^* (\tZ_1) - \tZ_1 \rVert, \ldots, \lVert x - \tZ_p \rVert \le \lVert \cX_{(M)}^*(\tZ_p) - \tZ_p \rVert \big)\\
  & ~\le \Big(\frac{N_0}{M}\Big)^p \P_{OW} \Big( \max_{j \in \zahl{p}} \nu_0(B_{\tZ_j,\lVert x - \tZ_j \rVert}) \le \max_{j \in \zahl{p}} \nu_0(B_{\tZ_j,\lVert \cX_{(M)}^*(\tZ_j) - \tZ_j \rVert}) \Big) \\
  & ~\le \Big(\frac{N_0}{M}\Big)^p \P_{OW} \Big( \frac{af_L}{f_U} \max_{j \in \zahl{p}} \nu_1(B_{x,\lVert x - \tZ_j \rVert}) \le \max_{j \in \zahl{p}} \nu_0(B_{\tZ_j,\lVert \cX_{(M)}^*(\tZ_j) - \tZ_j \rVert}) \Big)  \\
  & ~ \le \Big(\frac{f_U}{af_L}\Big)^p (1+o(1)) =O(1),
\end{align*}
where the asymptotic terms $o(1)$ and $O(1)$ depends only on $N_0, M$, and $p$. Therefore, by convexity of $x^p$ and following the same analysis in proof of Theorem \ref{thm:cons,BTlp}, the $L^p$ risk satisfies
\begin{align*}
    & \E_{OW} \Big[ \Big| \hat{r}_M^{*,+}(x) - r(x) \Big|^p \Big] \\
    & ~\lesssim \E_{OW} \Big[ \Big| \hat{r}_M^{*,+}(x) - \E_{OW}\big[ \hat{r}_M^{*,+}(x) \given \mX, \mW_{\mX} \big] \Big|^p \Big] + \E_{OW} \Big[ \Big| \E_{OW}\big[ \hat{r}_M^{*,+}(x) \given \mX, \mW_{\mX} \big] \Big|^p \Big] + \Big| r(x) \Big|^p \\
    & ~\lesssim \Big( \frac{N_0}{N_1M}\Big)^p \E_{OW} \Big[N_1^{p/2} \nu_1^{p/2}(A_M^{*,+}(x))\Big] + \Big( \frac{N_0}{M} \Big)^p \E_{OW} \Big[ \nu_1^p \big( A_M^{*,+}(x)\big) \Big] + \Big( \frac{f_U}{f_L} \Big)^p \lesssim 1,
\end{align*}
which holds uniformly for $x \in A_1(\epsilon)$. Consequently, by Fubini's theorem and Fatou's Lemma one has
\begin{align*}
    \E_{OW} \Big[ \int_{\bR^d} \Big\lvert \hat{r}_M^{*,+} - r(x) \Big\rvert^p \ind(x \in A_1(\epsilon) ) f_0(x) \d x \Big] \lesssim f_U \lambda(A_1(\epsilon)) \lesssim \epsilon.
\end{align*}

\vspace{0.2cm}

Combining the above three terms together, we obtain 
\begin{align*}
    & \lim_{N_0\to\infty} \E_{OW} \Big[ \int_{\bR^d} \Big\lvert \hat{r}_M^{*,+}(x) - r(x) \Big\rvert^p f_0(x) \d x \Big] \\
    & = \lim_{\epsilon \to 0} \lim_{N_0\to\infty} \E_{OW} \Big[ \int_{\bR^d} \Big\lvert \hat{r}_M^{*,+} - r(x) \Big\rvert^p \Big\{\ind(x \in A_1^c \cap A_2^c) + \ind(x \in A_1^c \cap A_2) + \ind(x \in A_1 ) \Big\} f_0(x) \d x \Big] \\
    & =0
\end{align*}
and thus finish the proof.
\end{proof}

\subsection{Proofs of lemmas in Section \ref{sec:proof-thm}}

This section gives proofs for the lemmas in Section \ref{sec:proof-thm}. Note that we can rewrite
\begin{align*}
    \hat{\tau}_{M}^{*,{\rm bc}} = \frac{1}{n}\sum_{i=1}^n \Big\{ \tilde{Y}_i^*(1) - \tilde{Y}_i^*(0) \Big\},
\end{align*}
where
\begin{align*}
\tilde{Y}_i^*(1) := 
    \begin{cases}
        Y_i^* & \mbox{ if } D_i^*=1, \\   
        \hat{\mu}_{1}(X_i) + M^{-1} \sum_{j \in \mathcal{J}_M^*(i)} \big( Y_j^* - \hat{\mu}_{1}(X_j^*) \big) & \mbox{ if } D_i^*=0
    \end{cases}    
\end{align*}
and
\begin{align*}
\tilde{Y}_i^*(0) := 
    \begin{cases}
         \hat{\mu}_{0}(X_i) + M^{-1} \sum_{j \in \mathcal{J}_M^*(i)} \big( Y_j^* - \hat{\mu}_{0}(X_j^*) \big) & \mbox{ if } D_i^*=1, \\   
         Y_i^* & \mbox{ if } D_i^*=0.
    \end{cases}    
\end{align*}
This observation will be used below.

\begin{proof}[Proof of Lemma \ref{lemma:BTest,rewrite}]
We give calculations for $\hat{\tau}_{M}^{*}$ and $\hat{\tau}_{M}^{*,{\rm bc}}$, respectively.

{\bf Part I.}
    For a unit such that $D_i=1$ and $W_{i}>0$, we define
    \begin{align*}
        K_M^*(i) := W_{i}^{-1} \sum_{j=1,D_j^* = 0}^n \sum_{\ell \in \mathcal{J}_M^*(j)} \ind(O_\ell^* = O_i)
    \end{align*}
    and let $K_M^*(i)$ be some arbitrary number between $K_M^{*,+}(i)$ and $K_M^{*,-}(i)$ if $W_{i}=0$. Then, by simple algebra and rearranging the indices, one has
    \begin{align*}
        \sum_{i=1,D_i^* = 0}^n \sum_{j \in \mathcal{J}_M^*(i)} Y_j^* &= \sum_{j=1,D_j^* = 0}^n \sum_{\ell \in \mathcal{J}_M^*(j)} Y_\ell^* = \sum_{j=1,D_j^* = 0}^n \sum_{\ell \in \mathcal{J}_M^*(j)} \Big\{ \sum_{i=1,D_i=1} \ind(O_\ell^* = O_i) Y_i \Big\} \\
        & =\sum_{i=1,D_i=1} \Big\{ \sum_{j=1,D_j^* = 0}^n \sum_{\ell \in \mathcal{J}_M^*(j)} \ind(O_\ell^* = O_i) \Big\} Y_i = \sum_{i=1,D_i=1} W_{i} K_M^*(i) Y_i.
    \end{align*}
    Employ the same argument while swapping $0$ and $1$. Notice that under this definition, $K_M^{*,-}(i) \le K_M^*(i) \le K_M^{*,+}(i)$ holds for all $i \in \zahl{n}$, since one can write
    \begin{align*}
        W_{i}Y_i K_M^{*}(i) = \sum_{j=1,D_j^* = 0}^n \sum_{\ell \in \mathcal{J}_M^*(j)} \ind(O_\ell^* = O_i) Y_\ell^*
    \end{align*}
    knowing that, given $Y_i \ge 0$, it holds true
    \begin{align*}
        W_{i}Y_i K_M^{*,-}(i) &= \sum_{j=1,D_j^* = 0}^n  \ind \big(\lVert X_i-X_j^* \rVert < \lVert \cX_{(M)}^*(X_j^*)-X_j^* \rVert \big) W_{i} Y_i \\
        & = \sum_{j=1,D_j^* = 0}^n \sum_{\ell=1,D_\ell^* = 1}^n \ind \big(\lVert X_\ell^*-X_j^* \rVert < \lVert \cX_{(M)}^*(X_j^*)-X_j^* \rVert \big) \ind(O_\ell^* = O_i) Y_\ell^* \\
        &\le \sum_{j=1,D_j^* = 0}^n \sum_{\ell \in \mathcal{J}_M^*(j)} \ind(O_\ell^* = O_i) Y_\ell^* = \sum_{j=1,D_j^* = 0}^n \sum_{\ell=1,D_\ell^* = 1}^n \ind\big(\ell \in \cJ_M^*(j)\big) \ind(O_\ell^* = O_i)  Y_\ell^* \\
        &\le \sum_{j=1,D_j^* = 0}^n \sum_{\ell=1,D_\ell^* = 1}^n \ind \big(\lVert X_\ell^*-X_j^* \rVert \le \lVert \cX_{(M)}^*(X_j^*)-X_j^* \rVert \big) \ind(O_\ell^* = O_i) Y_\ell^* \\
        &= \sum_{j=1,D_j^* = 0}^n  \ind \big(\lVert X_i-X_j^* \rVert \le \lVert \cX_{(M)}^*(X_j^*)-X_j^* \rVert \big) W_{i} Y_i = W_{i}Y_i K_M^{*,+}(i)
    \end{align*}
    for all units such that $D_i=1$ with $W_{i}>0$ and this applies too when we swap $0$ and $1$.
    Consequently, we can write
    \begin{align*}
        \hat{\tau}_{M}^{*} &= \frac{1}{n}\sum_{i=1}^n \big[ \hat{Y}_i^*(1) - \hat{Y}_i^*(0) \big] \\
        &= \frac{1}{n} \sum_{i=1,D_i^* = 1}^n Y_i^* + \frac{1}{nM} \sum_{i=1,D_i^* = 0}^n \sum_{j \in \mathcal{J}_M^*(i)} Y_j^* - \frac{1}{n} \sum_{i=1,D_i^* = 0}^n Y_i^* - \frac{1}{nM} \sum_{i=1,D_i^* = 1}^n \sum_{j \in \mathcal{J}_M^*(i)} Y_j^*  \\
        &= \frac{1}{n}\sum_{i=1}^n W_i (2D_i-1) Y_i +  \frac{1}{n}\sum_{i=1}^n W_i (2D_i-1) \frac{K_M^*(i)}{M} Y_i.
    \end{align*}
    
    \vspace{0.2cm}
    
{\bf Part II.}
    Similarly, we write
    \begin{align*}
        \hat{\tau}_{M}^{*,{\rm bc}} &= \frac{1}{n}\sum_{i=1}^n \big[ \tilde{Y}_i^*(1) - \tilde{Y}_i^*(0) \big] \\
        &= \frac{1}{n} \sum_{i=1,D_i^* = 1}^n Y_i^* + \frac{1}{n} \sum_{i=1,D_i^* = 0}^n  \hat{\mu}_{1}(X_i^*) + \frac{1}{nM} \sum_{i=1,D_i^* = 0}^n \sum_{j \in \mathcal{J}_M^*(i)} \big( Y_j^* - \hat{\mu}_{1}(X_j^*) \big) \\
        & ~~- \frac{1}{n} \sum_{i=1,D_i^* = 0}^n Y_i^* - \frac{1}{n} \sum_{i=1,D_i^* = 1}^n  \hat{\mu}_{0}(X_i^*) - \frac{1}{nM} \sum_{i=1,D_i^* = 1}^n \sum_{j \in \mathcal{J}_M^*(i)} \big( Y_j^* - \hat{\mu}_{0}(X_j^*) \big) \\
        &= \frac{1}{n} \sum_{i=1,D_i = 1}^n W_{i} Y_i + \frac{1}{n} \sum_{i=1,D_i = 0}^n W_{i}  \hat{\mu}_{1}(X_i) + \frac{1}{nM} \sum_{i=1,D_i = 1}^n W_i K_M^*(i) \big( Y_i - \hat{\mu}_{1}(X_i) \big)\\
        & ~~- \frac{1}{n} \sum_{i=1,D_i = 0}^n W_{i} Y_i - \frac{1}{n} \sum_{i=1,D_i = 1}^n  W_{i} \hat{\mu}_{0}(X_i) - \frac{1}{nM} \sum_{i=1,D_i = 0}^n W_i K_M^*(i) \big( Y_i - \hat{\mu}_{0}(X_i) \big).
    \end{align*}
    Adding and subtracting $n^{-1} [ \sum_{i=1,D_i = 1}^n W_{i}  \hat{\mu}_{1}(X_i) - \sum_{i=1,D_i = 0}^n W_{i}  \hat{\mu}_{0}(X_i) ]$, we then get 
    \begin{align*}
        \hat{\tau}_{M}^{*,{\rm bc}} = 
        \frac{1}{n} \sum_{i=1}^n W_i \big[\hat{\mu}_1(X_i) - \hat{\mu}_0(X_i)\big] + \frac{1}{n} \sum_{i=1}^n W_i (2D_i-1) \Big(1 + \frac{K^*_M(i)}{M}\Big) \big( Y_i - \hat{\mu}_{D_i}^*(X_i) \big),
    \end{align*}
    as desired. Trivially, from the second equality one can see this formulation agrees with the definition we give in Section \ref{sec:bootstrap} by simply writing 
    \begin{align*}
        \hat B_M^* &= \frac{1}{n}\sum_{i=1}^n\frac{2D_i^*-1}{M}\sum_{j\in\cJ^*_M(i)}\Big\{\hat\mu_{1-D_i^*}(X_i^*)-\hat\mu_{1-D_i^*}(X_j^*)\Big\} \\
        &=  \frac{1}{n} \sum_{i=1,D_i^* = 1}^n  \hat{\mu}_{0}(X_i^*) - \frac{1}{n} \sum_{i=1,D_i^* = 0}^n  \hat{\mu}_{1}(X_i^*) -\frac{1}{nM} \sum_{i=1,D_i^* = 1}^n \sum_{j \in \mathcal{J}_M^*(i)} \hat \mu_0(X_j^*) \\
        & ~~ + \frac{1}{nM} \sum_{i=1,D_i^* = 0}^n \sum_{j \in \mathcal{J}_M^*(i)} \hat \mu_1(X_j^*),
    \end{align*}
    which completes the proof.
\end{proof}

\begin{proof}[Proof of Lemma \ref{lemma:mbc,BTdist}]
    We write
    \begin{align*}
        & \Big(\frac{n_{1-D_i}}{M}\Big)^{\frac{p}{d}} \E_{OW} \Big[ \big\lVert U_{M,i}^* \big\rVert^p \Biggiven \mD\Big] \\
        &= \Big(\frac{n_{1-D_i}}{M}\Big)^{\frac{p}{d}} \int_{0}^\infty \P_{OW} \Big( \big\lVert \cX_{(M)}^*(X_i^*) - X_i^* \big\rVert \ge u \Biggiven \mD \Big) p u^{p-1} \d u \\
        &= \Big(\frac{1}{M}\Big)^{\frac{p}{d}} \int_{0}^\infty \P_{OW} \Big( \big\lVert \cX_{(M)}^*(X_i^*) - X_i^* \big\rVert \ge un_{1-D_i}^{-1/d} \Biggiven \mD \Big) p u^{p-1} \d u.  \yestag \label{eq:matching,boundedDist,integration}
    \end{align*}
Notice that, by Assumption \ref{asp:BT,Lp,risk}, one has for any $\omega \in \{0,1\}, x \in \cX$, 
\begin{align*}
    \nu_\omega\big(B_{x,u} \cap \cX \big) \ge f_L \lambda\big(B_{x,u} \cap \cX \big) \ge f_L a \lambda\big(B_{x,u} \big) = f_L a V_d u^d,
\end{align*}
where $V_d$ is the Lebesgue measure of the unit ball in $\bR^d$. Therefore, letting $c_0 := f_L a V_d$, we have
\begin{align*}
    & \P_{OW} \Big( \big\lVert \cX_{(M)}^*(X_i^*) - X_i^* \big\rVert \ge un_{1-D_i}^{-1/d} \Biggiven \mD, X_i^*=x \Big) \\
     =& \sum_{k=1}^{n_{1-D_i}} \P\Big( \big\lVert \cX_{(k)}(x) - x \big\rVert \ge un_{1-D_i}^{-1/d} \Biggiven \mD, \cX_{(M)}^*(x) = \cX_{(k)}(x), X_i^*=x \Big) p_k \\
     =& \sum_{k=1}^{n_{1-D_i}} \P\Big( \nu_{1-D_i}\big(B_{x,\lVert \cX_{(k)}(x) - x \rVert}\big) \ge \nu_{1-D_i}\big(B_{x,un_{1-D_i}^{-1/d}}\big) \Biggiven \mD, \cX_{(M)}^*(x) = \cX_{(k)}(x), X_i^*=x \Big) p_k \\
    \le& \sum_{k=1}^{n_{1-D_i}} \P\Big( U_{(k)} \ge \nu_{1-D_i}\big(B_{x,un_{1-D_i}^{-1/d} } \cap \cX \big) \Biggiven \mD, \cX_{(M)}^*(x) = \cX_{(k)}(x), X_i^*=x \Big) p_k \\
    \le& \sum_{k=1}^{n_{1-D_i}} \P\Big({\rm Bin}\big(n_{1-D_i}, c_0 u^d n_{1-D_i}^{-1} \big) \le k \Biggiven \mD, \cX_{(M)}^*(x) = \cX_{(k)}(x) \Big) p_k \\
    \le&  \sum_{k=1}^{n_{1-D_i}} \Big[\ind\big(k < c_0 u^d\big) \exp\Big\{ k - c_0 u^d + k \log\Big(\frac{c_0 u^d}{k}\Big)\Big\} + \ind\big(k \ge c_0 u^d\big) \Big] p_k  \yestag \label{eq:matching,boundedDist,unifbound}
\end{align*}
by Chernoff bound. Since this bound is uniform for $x \in \cX$, putting \eqref{eq:matching,boundedDist,integration} and \eqref{eq:matching,boundedDist,unifbound} together we have
\begin{align*}
    & \Big(\frac{n_{1-D_i}}{M}\Big)^{\frac{p}{d}} \E_{OW} \Big[ \big\lVert U_{M,i}^* \big\rVert^p \big| \mD\Big] \\
    \le& p M^{-\frac{p}{d}} \sum_{k=1}^{n_{1-D_i}} \int_{0}^\infty  \Big[\ind\big(k < c_0 u^d\big) \exp\Big\{ k - c_0 u^d + k \log\Big(\frac{c_0 u^d}{k}\Big)\Big\} + \ind\big(k \ge c_0 u^d\big) \Big] p_k u^{p-1} \d u \\
    =& \frac{p}{d} (c_0M)^{-\frac{p}{d}} \sum_{k=1}^{n_{1-D_i}} \Big[ \int_{0}^k u^{\frac{p}{d}-1} \d u + \Big(\frac{e}{k}\Big)^k \int_{k}^\infty u^{k + \frac{p}{d}-1} e^{-u} \d u \Big] p_k \\
    \le& c_0^{-\frac{p}{d}} \sum_{k=1}^{n_{1-D_i}} \Big(\frac{k}{M}\Big)^{\frac{p}{d}} p_k +  \frac{p}{d} (c_0M)^{-\frac{p}{d}} \sum_{k=1}^{n_{1-D_i}} \Big[ \Big(\frac{e}{k}\Big)^k \int_{k}^\infty u^{k + \frac{p}{d}-1} e^{-u} \d u \Big] p_k. \yestag \label{eq:matching,boundedDist,2terms}
\end{align*}

For the first term in \eqref{eq:matching,boundedDist,2terms}, with exactly the same analysis as in \eqref{eq:matching,DRlocal,boundedSum}, we know it is bounded by, say,  $C_p'$. 

For the second term in \eqref{eq:matching,boundedDist,2terms}, for any $k \to \infty$ and $k \le M$ we have, by Stirling's formula,
\begin{align*}
    & M^{-\frac{p}{d}} \Big(\frac{e}{k}\Big)^k \int_{k}^\infty u^{k + \frac{p}{d}-1} e^{-u} \d u = M^{-\frac{p}{d}} \Big(\frac{e}{k}\Big)^k \Gamma \Big( k+\frac{p}{d}, k\Big) \le M^{-\frac{p}{d}} \Big(\frac{e}{k}\Big)^k \Gamma \Big( k+\frac{p}{d}\Big) \\
    & \asymp M^{-\frac{p}{d}} \Big(\frac{e}{k}\Big)^k \Big(\frac{k}{e}\Big)^{\frac{p}{d}-1} \Gamma(k+1) \asymp M^{-\frac{p}{d}} \Big(\frac{e}{k}\Big)^k \Big(\frac{k}{e}\Big)^{\frac{p}{d}-1} \sqrt{2\pi k} \Big(\frac{e}{k}\Big)^{-k} \asymp \Big(\frac{k}{M}\Big)^{\frac{p}{d}} \sqrt{\frac{2\pi}{k}} = o(1); \yestag \label{eq:matching,boundedDist,Stirling}
\end{align*}
here, we have used $``\asymp"$ to denote asymptotic equivalence. Then, this sequence in $k$ is upper bounded by some constant $B_p$. For any $K \le M$, \eqref{eq:matching,boundedDist,2terms} is then upper bounded by
\begin{align*}
    \Big(\frac{n_{1-D_i}}{M}\Big)^{\frac{p}{d}} \E_{OW} \Big[ \big\lVert U_{M,i}^* \big\rVert^p \big| \mD\Big] &\le
    C_p' + \frac{p}{d} c_0^{-\frac{p}{d}} B_p \sum_{k \le K} p_k + \frac{p}{d} c_0^{-\frac{p}{d}} \frac{\sqrt{2\pi}}{M^{1/4}} \exp\Big\{\frac{1}{12M^{1/2}}\Big\} \sum_{k \le K} \Big(\frac{k}{M}\Big)^{\frac{p}{d}} p_k \\
    & ~~+ \frac{p}{d} (c_0M)^{-\frac{p}{d}} \sum_{k>K} \Big[ \Big(\frac{e}{k}\Big)^k \int_{k}^\infty u^{k + \frac{p}{d}-1} e^{-u} \d u \Big] p_k.
\end{align*}
Setting $K = M^{1/2}$, since applying \eqref{eq:matching,boundedDist,Stirling} on those terms with $k \ge K+1$ implies we must also have
\begin{align*}
    \frac{p}{d} (c_0M)^{-\frac{p}{d}} \sum_{k>K} \Big[ \Big(\frac{e}{k}\Big)^k \int_{k}^\infty u^{k + \frac{p}{d}-1} e^{-u} \d u \Big] p_k \le\frac{p}{d} c_0^{-\frac{p}{d}} \frac{\sqrt{2\pi}}{M^{1/4}} \exp\Big\{\frac{1}{12M^{1/2}}\Big\}  \sum_{k>K} \Big(\frac{k}{M}\Big)^{\frac{p}{d}} p_k,
\end{align*}
we can then conclude
\begin{align*}
    \Big(\frac{n_{1-D_i}}{M}\Big)^{\frac{p}{d}} \E_{OW} \Big[ \big\lVert U_{M,i}^* \big\rVert^p \big| \mD\Big] \le C_p' (1+o(1))
\end{align*}
by noting Chernoff bound applied to the upper tail probability gives
\begin{align*}
    \sum_{k \le K} p_k = \P \Big( {\rm Bin}\Big(N_0, \frac{K}{N_0}\Big) > M \Big) \le \exp\Big\{ K \Big[ \frac{M}{K} - 1 - \frac{M}{K} \log\Big(\frac{M}{K}\Big) \Big] \Big\} \le \exp\{-M\}.
\end{align*}
It then yields the asserted inequality \eqref{eq:matching,boundedDist,goal} by taking $C_p := 2C_p'$.
\end{proof}

\subsection{Auxiliary lemmas}\label{sec:aux}

\begin{lemma}[Lemma S3.1, \cite{lin2023supp}]\label{lemma:leb,p}
  Let $\nu$ be a probability measure on $\bR^d$ admitting a Lebesgue density $f$. Let $x \in \bR^d$ be a Lebesgue point of $f$. We then have, for any $\epsilon \in (0,1)$, there exists $\delta = \delta_x>0$ such that for any $z \in \bR^d$ such that $\lVert z - x \rVert \le \delta$, we have
  \begin{align*}
    \Big\lvert \frac{\nu(B_{x,\lVert z-x \rVert})}{\lambda(B_{x,\lVert z-x \rVert})} - f(x) \Big\rvert \le \epsilon,~~ \Big\lvert \frac{\nu(B_{z,\lVert z-x \rVert})}{\lambda(B_{z,\lVert z-x \rVert})} - f(x) \Big\rvert \le \epsilon.
  \end{align*}
\end{lemma}

\begin{lemma}
[Central moments of the bootstrap binomial distribution]
\label{lemma:moment, bootstrapBinomial}

Suppose that $\{X_i\}_{i=1}^n \overset{i.i.d.}{\sim} {\rm Ber}(p)$ and 
the boostrap weights $W=(W_1, \ldots, W_n)' \sim {\rm Mult}(n, 1/n, \ldots, 1/n )$ independent of $\{X_i\}_{i=1}^n$. Then for any positive integer $k$ and $Y := \sum_{i=1}^n W_i X_i$, its $k$-th central moments $\mu_k^*$ can be written into the following forms for odd and even $k$ values:
\begin{align*}
    \mu_{2r}^* &= P_{1,2r}^* np(1-p) + P_{2,2r}^* [np(1-p)]^2 + \ldots + P_{r,2r}^* [np(1-p)]^r \\
    \mu_{2r+1}^* &= P_{1,2r+1}^* np(1-p) + P_{2,2r+1}^* [np(1-p)]^2 + \ldots + P_{r,2r+1}^* [np(1-p)]^r,
\end{align*}
where $P_{u,v}^*$ denotes certain polynomial in $p$ and $1-p$ that does not contain $n$.
\end{lemma}
\begin{proof}
For any positive integer $k$, from multinomial theorem we can write
\begin{align*}
    \mu_{k}^* &= \E_{OW} \Big[ \sum_{i=1}^n W_i X_i - np \Big]^k = \E_{OW} \Big[ \sum_{i=1}^n W_i (X_i - p) \Big]^k \\
    &= \sum_{ \{k_1, \ldots, k_m\} \in \mathcal{P}(k)} \frac{n (n-1) \cdots (n-m+1)}{s(k_1, \ldots, k_m)} \E \Big[ \prod_{l=1}^m W_l^{k_l} \Big] \prod_{l=1}^m \E[X_1 - p]^{k_l},
\end{align*}
where $\mathcal{P}(k)$ contains all partitions of $k$ into positive integers and $s(k_1, \ldots, k_m)$ denotes the number of repetitive elements in a partition $\{k_1, \ldots, k_m\}$. These two numbers are influenced by $k$ and do not depend on $n$. 

Next, for the second product term in the summand, we observe that the moment generating function of $W_1$ is 
\[
M_{W_1}(t) = \Big[1-\frac{1}{n} + \frac{1}{n} e^t\Big]^n =: V^n. 
\]
With a slight abuse of notation, assume for some $k \in \bN$ with $N_\ell \lesssim 1$ for all $\ell \in \zahl{k}$, 
\begin{align*}
    \frac{\partial^k}{\partial t^k} M_{W_1}(t) = N_{k} V^{n-k} e^{kt} + N_{k-1} V^{n-(k-1)} e^{(k-1)t} + \ldots + N_{1} V^{n-1} e^t,
\end{align*}
as it is certainly the case for $\frac{\partial}{\partial t} M_{W_1}(t)$. We can then get
\begin{align*}
    \frac{\partial^{(k+1)}}{\partial t^{(k+1)}} M_{W_1}(t) 
    &= N_{k+1}' V^{n-(k+1)} e^{(k+1)t} + N_{k}' V^{n-k} e^{kt} + \ldots + N_{1}' V^{n-1} e^t
\end{align*}
as with $N_{k+1}=N_{0}=0$, we have $N_\ell' = \ell N_\ell + \frac{n-(\ell-1)}{n} N_{\ell-1} \lesssim 1$ for all $\ell \in \zahl{k+1}$. By induction, this implies for any $k<\infty$,
\begin{align*}
    \E\big[W_1^k\big] = \frac{\partial^k}{\partial t^k} M_{W_1}(t) \Big|_{t=0} \lesssim 1.
\end{align*}
Again, by induction on $m$ and using Hölder's inequality, we immediately get, for any $k_1, \ldots, k_m < \infty$, 
\[
\E\Big[\prod_{\ell=1}^m W_l^{k_\ell}\Big] \lesssim 1. 
\]
From \cite{romanovsky1923note}, we know the central moments $\{\mu_k\}_{k=1}^{\infty}$ for ${\rm Ber}(p)$ can be written as a reduction formula
\begin{align*}
    \mu_{k+1} = p(1-p)\Big(k\mu_{k-1}+\frac{\d \mu_k}{\d p}\Big),
\end{align*}
which implies the following representations of the odd and even central moments: 
\begin{align*}
    \mu_{2r} &= P_{1,2r} p(1-p) + P_{2,2r} [p(1-p)]^2 + \ldots + P_{r,2r} [p(1-p)]^r \\
 {\rm and}~~~   \mu_{2r+1} &= P_{1,2r+1} p(1-p) + P_{2,2r+1} [p(1-p)]^2 + \ldots + P_{r,2r+1} [p(1-p)]^r,
\end{align*}
where $P_{u,v}$ denotes certain polynomials in $p$ and $1-p$. Note since $\mu_1=0$, the polynomial $n (n-1) \cdots (n-m+1)$ are at most of order $\floor{n/2}{}$ and we get the desired representation.
\end{proof}

\begin{lemma}[Hardy–Littlewood maximal inequality, \cite{stein2009real}]\label{lemma:HL}
  For any locally integrable function $f:\bR^d \to \bR$, we define its maximal function ${\sf M}f(x)$ by
  \[
    {\sf M}f(x) := \sup_{x \in B} \frac{1}{\lambda(B)} \int_{B} \lvert f(z) \rvert \d z, ~~~ x \in \bR^d,
  \]
  where the supremum is taken over all balls containing $x$.
  Then for $d \ge 1$, there exists a constant $C_d>0$ only depending on $d$ such that for all $t>0$ and $f \in L_1(\bR^d)$, we have
  \[
    \lambda(\{x \in \bR^d :{\sf M}f(x) > t\}) < \frac{C_d}{t} \lVert f \rVert_{L_1},
  \]
  where $\lVert \cdot \rVert_{L_1}$ stands for the function $L_1$ norm.
\end{lemma}

\begin{lemma}[Transition of stochastic orders, modified from Lemma 3 in \cite{cheng2010bootstrap}]\label{lemma:StoOrders}
  Consider the product probability space given in  \eqref{eq:setup,ProbSpace}. For a random quantity $R_n$, it holds that $R_n = o_{\P_{OW}}(1)$ if and only if $R_n=o_{\P_W}(1)$ in $\P_{O}$-probability. Moreover, when a random quantity $Q_n$ has a bootstrap distribution that asymptotically imitates an unconditional distribution of $Q$, which has a continuous cumulative distribution function, that is,
  \begin{align*}
      \sup_{t \in \bR} \big\lvert \P_{\mW|\mO}\big(Q_n \le t \big) - \P(Q \le t) \big\rvert &= o_{\P_{O}}(1)
  \end{align*}
  then we also have the same for $Q_n + R_n$, provided that $R_n=o_{\P_W}(1)$ in $\P_{O}$-probability:
  \begin{align*}
       \sup_{t \in \bR} \big\lvert \P_{\mW|\mO}\big(Q_n+R_n \le t \big) - \P(Q \le t) \big\rvert &= o_{\P_{O}}(1). \yestag \label{eq:StoOrd,op1,Slutsky}
  \end{align*}
\end{lemma}

\begin{proof}
    {\bf Part I.} We first prove 
    \begin{align*}
        R_n = o_{\P_{OW}}(1) \text{ if and only if } R_n=o_{\P_W}(1) \text{ in $\P_{O}$-probability}.
    \end{align*}
    For all $\epsilon,\delta>0$, by Markov's inequality one has
    \begin{align*}
        \P_{O}(  \P_{W|O}(\lvert R_n \rvert \ge \epsilon) \ge \delta ) & \le \frac{1}{\delta} \E_{O}\big[\P_{W|O}(\lvert R_n \rvert  \ge \epsilon) \big] = \frac{1}{\delta} \E_{O}\big[\E_{W|O}\{\ind(\lvert R_n \rvert \ge \epsilon)\} \big] \\
        & =\frac{1}{\delta} \E_{OW} [\ind(\lvert R_n \rvert \ge \epsilon)] = \frac{1}{\delta}\P_{OW}(\lvert R_n \rvert \ge \epsilon).
    \end{align*}
    Therefore we conclude that $R_n = o_{\P_{OW}}(1)$ would imply $R_n=o_{\P_W}(1)$ in $\P_{O}$-probability. The other direction follows from, for any $\epsilon>0$ and an arbitrary $\eta>0$,
    \begin{align*}
        \P_{OW}(\lvert R_n \rvert \ge \epsilon) &= \E_{O}\big[\P_{W|O}(\lvert R_n \rvert \ge \epsilon) \big] \\ 
        &= \E_{O}\big[\P_{W|O}(\lvert R_n \rvert \ge \epsilon) \ind \big(\P_{W|O}(\lvert R_n \rvert \ge \epsilon) \ge \eta \big) \big] \\
        & ~~ + \E_{O}\big[\P_{W|O}(\lvert R_n \rvert \ge \epsilon) \ind \big(\P_{W|O}(\lvert R_n \rvert \ge \epsilon) < \eta \big) \big] \\
        &\le \E_{O}\big[ \ind \big(\P_{W|O}(\lvert R_n \rvert \ge \epsilon) \ge \eta \big) \big] + \eta \\
        &= \P_{O}\big( \P_{W|O}(\lvert R_n \rvert \ge \epsilon) \ge \eta \big) + \eta.
    \end{align*}
    Of note, any random element $R_n$ that is defined only on $( \Omega_O^\infty, \cA_O^\infty, \P_O^\infty )$ with stochastic order $o_{\P_O}(1)$ is also of of an order $o_{\P_W}(1)$ in $\P_O$-probability. As for any $\epsilon, \delta>0$, we have
\begin{align*}
    \P_O \big( \P_{W|O}(\lvert R_n \rvert > \epsilon) > \delta \big) &\le \frac{1}{\delta} \E_O \big[ \E_{W|O}\{\ind(\lvert R_n \rvert > \epsilon) \}\big] = \frac{1}{\delta} \E_O \big[ \ind(\lvert R_n \rvert > \epsilon)\big] =\frac{1}{\delta} \P_O (\lvert R_n \rvert > \epsilon)
\end{align*}
    by Markov's inequality, and the first equality follows from $R_n$ does not depend on the bootstrap weights.
    
    {\bf Part II.} Now we proceed to prove the second statement, i.e. $R_n$ does not have an impact on the bootstrap distribution of $Q_n$, asymptotically speaking. On one hand, we have
    \begin{align*}
         \P_{W|O}(Q_n+R_n \le t) &= \P_{W|O}(Q_n+R_n \le t, \lvert R_n \rvert \ge \epsilon) + \P_{W|O}(Q_n+R_n \le t, \lvert R_n \rvert < \epsilon) \\
         &\le \P_{W|O}(\lvert R_n \rvert \ge \epsilon) + \P_{W|O}(Q_n \le t + \epsilon).
    \end{align*}
    On the other hand, we have
    \begin{align*}
        \P_{W|O}(Q_n+R_n \le t) &= 1-\P_{W|O}(Q_n+R_n> t) \\
        &=  1-\P_{W|O}(Q_n+R_n> t, \lvert R_n \rvert < \epsilon) -\P_{W|O}(Q_n+R_n> t, \lvert R_n \rvert \ge \epsilon) \\
        & \ge 1-\P_{W|O}(Q_n> t-\epsilon) -\P_{W|O}( \lvert R_n \rvert \ge \epsilon).
    \end{align*}
    Combining the above two facts yields that, for any $t \in \bR$, 
    \begin{align*}
        o_{\P_{O}}(1) + \P_{W|O}(Q_n \le t-\epsilon) \le \P_{W|O}(Q_n+R_n \le t) \le \P_{W|O}(Q_n \le t + \epsilon) + o_{\P_{O}}(1).
    \end{align*}
    Then at all $t$ such that $ \P(Q \le t)$ is continuous at $t$, we can take $\epsilon \downarrow 0$ and have $\P_{W|O}(Q_n+R_n \le t) = \P(Q \le t) + o_{\P_{O}}(1)$. The uniform result in \eqref{eq:StoOrd,op1,Slutsky} follows from Polya's theorem.
\end{proof}

\begin{lemma}[Marcinkiewicz-Zygmund strong law of large numbers, \cite{durrett2019probability}] 
\label{lemma:MZ-SLLN}

Let $\{X_i\}_{i=1}^n$ be i.i.d. random variables with $\E[\lvert X_1 \rvert^p] < \infty $ for some $0 < p < 2$, then
\begin{align*}
    \begin{cases}
        \frac{S_n - n\E[X]}{n^{1/p}} \stackrel{\sf a.s.}{\longrightarrow} 0, & \mbox{ if } 1<p<2, \\   
        \frac{S_n}{n^{1/p}} \stackrel{\sf a.s.}{\longrightarrow} 0, & \mbox{ if } 0<p<1;
    \end{cases}    
\end{align*}
where $S_n = \sum_{i=1}^n X_i$.
\end{lemma}

{
\bibliographystyle{apalike}
\bibliography{AMS}
}

\end{document}